\documentclass[11pt,a4paper]{article}
\usepackage[utf8]{inputenc}
\usepackage{amsmath,amssymb,amsfonts,stmaryrd}
\usepackage{enumerate}
\usepackage[margin=3cm]{geometry}
\usepackage{siunitx}
\hypersetup{
	colorlinks=true,        
	linkcolor=blue,         
	citecolor=darkgreen,         
	urlcolor=blue}           
\usepackage{graphicx}
\usepackage{mathtools}
\usepackage{array}
\usepackage{subcaption}
\usepackage{todonotes}
\usepackage{algorithm}
\usepackage{algorithmicx}
\usepackage{algpseudocode}

\definecolor{darkgreen}{RGB}{0,150,0}
\definecolor{rust}{rgb}{0.6,0.1,0.1}
\usepackage{amsthm}
\usepackage{thmtools}
\newtheorem{theorem}{Theorem}[section]
\newtheorem{lemma}[theorem]{Lemma}
\newtheorem{proposition}[theorem]{Proposition}
\newtheorem{corollary}[theorem]{Corollary}
\newtheorem{definition}[theorem]{Definition}

\newtheorem{problem}[theorem]{Problem}
\theoremstyle{remark}
\newtheorem{remark}[theorem]{Remark}
\newtheorem{example}[theorem]{Example}

\DeclareMathOperator{\Id}{Id}

\DeclareMathOperator{\dom}{dom}
\DeclareMathOperator{\sign}{sign}

\DeclareMathOperator{\gra}{gra}
\DeclareMathOperator{\zer}{zer}
\DeclareMathOperator{\Fix}{Fix}
\DeclareMathOperator{\prox}{prox}

\DeclareMathOperator{\rank}{rank}

\newcommand{\Hilbert}{\mathcal{H}}

\newcommand{\Gilbert}{\mathcal{G}}
\newcommand{\R}{\mathbb{R}}

\newcommand{\Primal}{\mathcal{P}}
\newcommand{\Dual}{\mathcal{D}}
\newcommand{\setto}{\rightrightarrows}

\newcommand*\colvec[1]{\begin{pmatrix}#1\end{pmatrix}}

\newcommand*\barbf[1]{\bar{\mathbf{#1}}}

\title{A primal-dual splitting algorithm for composite monotone inclusions with minimal lifting}

\author{Francisco J.\ Arag\'on-Artacho\thanks{Department of Mathematics,
                             University of Alicante,
                             Alicante, \textsc{Spain}.
                                 Email:~\href{mailto:francisco.aragon@ua.es}
                                 {francisco.aragon@ua.es}}
        \and
       Radu I. Bo\c{t}\thanks{Faculty of Mathematics,
       			 University of Vienna,
       			 Vienna, \textsc{Austria}.
       			 Email:~\href{mailto:radu.bot@univie.ac.at}
       			 {radu.bot@univie.ac.at}}	
         \and
        David Torregrosa-Bel\'en\thanks{Department of Mathematics,
                             University of Alicante,
                             Alicante, \textsc{Spain}.
                                 Email:~\href{mailto:david.torregrosa@ua.es}
                                 {david.torregrosa@ua.es}}
 }

\graphicspath{{Figures/}}
\begin{document}

\maketitle

\begin{abstract}
In this work, we study resolvent splitting algorithms for solving composite monotone inclusion problems. The objective of these general problems is finding a zero in the sum of maximally monotone operators composed with linear operators. Our main contribution is establishing the first primal-dual splitting algorithm for composite monotone inclusions with minimal lifting. Specifically, the proposed scheme reduces the dimension of the product space where the underlying fixed point operator is defined, in comparison to other algorithms, without requiring additional evaluations of the resolvent operators. We prove the convergence of this new algorithm and analyze its performance in a problem arising in image deblurring and denoising. This work also contributes to the theory of resolvent splitting algorithms by extending the minimal lifting theorem  recently proved by Malitsky and Tam to schemes with resolvent parameters.
\end{abstract}


\section{Introduction}

In the last decades, monotone inclusion problems have become an attractive topic of research in operator theory and numerical optimization. The wide variety of situations in applied mathematics that can be modeled as finding a zero of the sum of mixtures of maximally monotone operators is one of the reasons for its increasing popularity. Among the methods that are usually employed for tackling these problems, \emph{splitting algorithms} (see, e.g.,~\cite[Chapter~26]{bauschke2017}) are the ones that have received more attention. Using simple operations, these methods define an iterative sequence which separately handles the operators defining the problem and is convergent to a solution to the inclusion problem. Further, as these methods only use first-order information, they are well suited for large-scale optimization problems.

In this work, we focus on the study of \emph{primal-dual splitting algorithms} for composite monotone inclusion problems in real Hilbert spaces of the following form.
\begin{problem}\label{problem:PD}
Let $\Hilbert$ and $(\Gilbert_j)_{1\leq j\leq m}$ be real Hilbert spaces. Let $A_1, \ldots, A_n:\Hilbert \setto \Hilbert$ be maximally monotone operators, let $B_j:\Gilbert_j\setto\Gilbert_j$ be maximally monotone and $L_j:\Hilbert \to\Gilbert_j$ be a bounded linear operator whose adjoint is denoted by $L_j^*$, for all $j\in{\{1,\ldots,m\}}$ . The problem consists in solving the  primal inclusion
\begin{equation}\label{eq:primalintroBj}
    \text{find } x \in \Hilbert \text{ such that } 0\in\sum_{i=1}^n A_i(x) + \sum_{j=1}^{m}L_j^*B_j (L_jx),
\end{equation}
together with its associated dual inclusion
\begin{equation}\label{eq:dualintroBj}
    \text{find } (u_1,\ldots,u_m)\in\Gilbert_1\times\cdots\times\Gilbert_m \text{ such that } \left(\exists\,x\in\Hilbert\right)
    \left\{
    \begin{aligned}
    & -\sum_{j=1}^{m} L_j^*u_j \in \sum_{i=1}^n A_i(x), & \\
   & u_j \in B_j(L_jx)   \quad \forall j\in{\{1,\ldots,m\}}.
    \end{aligned}
    \right.
\end{equation}
\end{problem}
Problem~\ref{problem:PD} encompasses numerous important problems in mathematical optimization and real-world applications, see e.g.~\cite{chambollelions, chambollepock,setzer2011}. In these settings, it is highly desirable to devise algorithms that simultaneously  obtain  solutions to both problems~\eqref{eq:primalintroBj} and~\eqref{eq:dualintroBj} --namely, a \emph{primal-dual solution}-- and which only make use of resolvents of the maximally monotone operators, forward evaluations of the linear operators and their adjoints, scalar multiplication and vector addition. Many splitting methods can be found in the literature satisfying these conditions, see e.g.~\cite{ bot2013, bot2015,botDR, combettespesquet,vu2013}. One of the best-known primal-dual algorithm is the one proposed by Brice\~{n}o-Arias and Combettes in~\cite{bricenocombettes2011}, which was further studied in~\cite{bot2016}. To derive this scheme, let us consider first the particular instance of Problem~\ref{problem:PD} in which $n=m=1$ and let us define the pair of operators $M$ and $N$ given by
\begin{equation*}\label{eq:fbf}
\left\{
\begin{aligned}
M:\Hilbert\times\Gilbert\setto\Hilbert\times\Gilbert :& (x,u) \to A(x) \times B^{-1}(u), \\
N:\Hilbert\times\Gilbert\to\Hilbert\times\Gilbert : &(x,u) \to (L^*u,-Lx).
\end{aligned}
\right.
\end{equation*}
The operator $M$ is maximally monotone and $N$ is a skew symmetric bounded linear operator. Further, the set of zeros of the sum $M+N$ consists of primal-dual solutions to Problem~\ref{problem:PD}. Applying the forward-backward-forward algorithm to the problem of finding the zeros of $M+N$ results in the fixed point iteration given by
\begin{equation}\label{eq:fbfit}
x^{k+1} = \big( J_{ \gamma  M} \left(\Id - \gamma N\right) +  \gamma N \left( \Id - J_{\gamma M}\left( \Id - \gamma N\right)\right)\big)(x^k) \quad \forall k \geq 0,
\end{equation}
where $\gamma>0$, $\Id$ denotes the identity operator and $J_{\gamma A}$ stands for the resolvent of $A$ with parameter $\gamma$ (see Definition~\ref{def:resolvent}). Thus, since the resolvent of a cartesian product is the cartesian product of the resolvents, it can be seen that~\eqref{eq:fbfit} is a \emph{full splitting algorithm}, as it only requires evaluations of the resolvents $J_{\gamma A}$ and $J_{\gamma B^{-1}}$, and of the linear operator and its adjoint.

The general problem involving more than two operators can be addressed by setting
\begin{equation*}
A :=A_1, B  := A_2\times\cdots\times A_n \times B_1\times\cdots\times B_m \text{ and } L := \Id  \times \stackrel{(n)}{\cdots} \times  \Id \times L_1 \times \cdots \times L_m.
\end{equation*}
In this case, according to~\eqref{eq:fbfit}, the resulting algorithm is generated by a fixed point iteration of an operator defined in the ambient space $\Hilbert^{n} \times \Gilbert_1 \times \cdots \times \Gilbert_m$.
The dimension of the underlying space is directly related to the memory requirements of the resulting algorithm. In general, a smaller dimension of the space translates into less consumption of computational resources. For this reason, the development of  algorithms with reduced dimension for solving monotone inclusion problems has recently become an active topic of research~\cite{Campoy21,malitsky2021resolvent,TamEtAl21,ryu2020}.

\paragraph{Lifted splitting algorithms} The notion of \emph{lifted splitting}, first introduced in~\cite{ryu2020},  relates a fixed point algorithm with the dimension of its underlying ambient space. Consider the simplest case of  the classical monotone inclusion problem obtained by setting $m=0$ in~\eqref{eq:primalintroBj}:
\begin{problem}\label{eq:monotoneinclusion}
 Let  $A_1,\ldots,A_n:\Hilbert\setto\Hilbert$ be maximally monotone operators and consider the problem
 \begin{equation*}
\text{ find } x\in\Hilbert \text{ such that } 0\in\sum_{i=1}^n A_i(x).
\end{equation*}
\end{problem}
A fixed point algorithm for finding a solution to Problem~\ref{eq:monotoneinclusion} employs a $d$-\emph{fold lifting} if its underlying fixed point operator can be defined on the $d$-fold Cartesian product $\Hilbert^{d}$. For example, if $n=2$, the famous Douglas--Rachford algorithm~\cite{lionsmercier} makes use of a 1-fold lifting, since it can be written as the fixed point iteration
\begin{equation*}
x^{k+1} = x^k + \lambda \left( J_{A_2}\left( 2J_{A_1} -\Id\right) -J_{A_1} \right)(x^k)  \quad \forall k \geq 0,
\end{equation*}
with $\lambda\in{]0,2[}$. Until very recently, the only way to tackle the problem when $n>2$ was using Pierra's product space reformulation~\cite{Pierra}, which implies an $n$-fold lifting. Nowadays, various algorithms have been proposed
allowing to solve the problem by only resorting to an $(n-1)$-lifting, see e.g.~\cite{Campoy21,TamEtAl21}. This reduction from $n$ to $n-1$ has been proven to be minimal~\cite{malitsky2021resolvent}
when the algorithms are required to be~\emph{frugal resolvent splittings}~\cite{ryu2020}, which means that
each of the resolvents $J_{A_1}, \ldots, J_{A_n}$ is evaluated only once per iteration. 

To the best of the authors' knowledge, the notion of lifting has not been developed in the setting of primal-dual inclusions given by Problem~\ref{problem:PD}. We will say that a primal-dual splitting has \emph{$(d,f)$-lifting} if the underlying fixed point operator can be written in the product space
\begin{equation*}
 \Hilbert^{d} \times \Gilbert_{1}^{f_1} \times \cdots \times \Gilbert_{m}^{f_m},
\end{equation*}
with $f = \sum_{j=1}^m f_j$. Thus, the Briceño-Arias--Combettes primal-dual splitting algorithm makes use of an $(n,m)$-fold lifting. This is also the case for the other primal-dual algorithms existing in the literature. In this work, we propose the first $(n-1,m)$-lifted splitting method for solving primal-dual inclusions and demonstrate the minimality of the algorithm. In order to do this, it is important to note the definition of frugal resolvent splitting does not allow the use of parametrized resolvents. The inclusion  of these resolvent parameters is of crucial importance for controlling the Lipschitz constants of the linear operators in~Problem~\ref{problem:PD}, as can be seen in all the existent primal-dual schemes. This motivates the introduction of the concept of \emph{frugal parametrized resolvent splitting} whose definition coincides with the one of frugal resolvent splitting except that it permits the inclusion of resolvent parameters. Our contribution to the theory of minimal lifting splitting methods is double: (i) we extend the results of Malitsky--Tam in~\cite[Section~3]{malitsky2021resolvent} to frugal parametrized resolvent splitting algorithms, (ii) we prove that for a frugal primal-dual parametrized resolvent splitting (see Section~\ref{sect:minimal} for a precise definition) with $(d,m)$-fold lifting to solve Problem~\ref{problem:PD}, one necessarily has $d\geq n-1$. Our proposed algorithm is the first\footnote{The method recently proposed by Brice\~{n}o-Arias in~\cite{Briceno21}, which differs from ours in the last component of the vector $\mathbf{x}$, is not correct. For instance, if, in the setting of Problem 1.1 in~\cite{Briceno21}, ${\cal H} = {\cal G} = \R^2$, $A_1=A_3=0$, $A_2=\Id-(1,1)^T$ and $L=\begin{psmallmatrix}1 & 0\\ 0 & 0\end{psmallmatrix}$, then $(z, u)=(1,1/2,1,0)^T$ is a fixed point of the underlying operator with $\delta=1$ and $\gamma >0$, but $x=J_{A_1}(z)=z = (1,1/2)^T$ is not a zero of the sum, since $A_1(x)+A_2(x)+L^*A_3(Lx)=(0,-1/2)^T$.
} algorithm in the literature being minimal according to this relation.

The rest of this work is structured as follows. In Section~\ref{sect:2} we recall some preliminary notions and results. In particular, in Section~\ref{subsect:resolventsplitting} we present the extension of the results by Malitsky--Tam~\cite{malitsky2021resolvent} to parametrized resolvent splitting algorithms. In Section~\ref{sect:alg}, we introduce the first primal-dual algorithm with reduced lifting for composite monotone inclusion problems and prove its convergence. The concept of parametrized resolvent splitting is adapted to primal-dual schemes in Section~\ref{sect:minimal}. We prove a minimality theorem under the hypothesis of frugality and show that our proposed algorithm verifies it. In Section~\ref{sect:numerics} we include a numerical experiment on image deblurring and compare the performance of the new algorithm with the best performing primal-dual algorithm for this problem. The paper ends with some conclusions and possible future work directions in Section~\ref{sect:conclusions}. Finally, in Appendix~\ref{sect:appendix}, a detailed proof of the results in Section~\ref{subsect:resolventsplitting} is presented.

\section{Preliminaries}\label{sect:2}

Throughout this paper, $\Hilbert$, $\Gilbert$ and $(\Gilbert_j)_{1 \leq j \leq m}$ are real Hilbert spaces. Otherwise stated, to simplify the notation we will employ  $\langle \cdot,\cdot\rangle$ and $\|\cdot\|$ to denote the inner product and the induced norm, respectively, of any space. We use $\to$ to denote norm convergence of a sequence. We denote by $\Hilbert^n$ the product Hilbert space  $\Hilbert^n=\Hilbert \times \stackrel{(n)}{\cdots} \times \Hilbert$  with inner product defined as
\begin{equation*}
\langle (x_1,\ldots,x_n), (\bar{x}_1,\ldots,\bar{x}_n)\rangle := \sum_{i=1}^n \langle x_i,\bar{x}_i\rangle\quad\forall (x_1,\ldots,x_n), (\bar{x}_1,\ldots,\bar{x}_n)\in\Hilbert^n.
\end{equation*}
Sequences and sets in product spaces are marked with bold, e.g., $\mathbf{x}=(x_1,\ldots,x_n)\in\Hilbert^n$.

For a \emph{set-valued operator}, we write $A:\Hilbert \setto\Hilbert$, in opposite to $A:\Hilbert\to\Hilbert$ which denotes a \emph{single-valued operator}. The  notation $\dom$, $\Fix$, $\zer$ and $\gra$ is used for the \emph{domain}, the \emph{set of fixed points}, the \emph{zeros}  and the \emph{graph} of $A$, respectively, i.e.,
\begin{align*}
\dom A&:=\left\{x\in\Hilbert : A(x)\neq\emptyset\right\}, &
\Fix A&:=\left\{x\in\Hilbert : x\in A(x)\right\}, \\
\zer A&:=\left\{x\in\Hilbert : 0\in A(x)\right\}, & \gra A&:=\left\{(x,u)\in\Hilbert\times\Hilbert : u\in A(x)\right\} .
\end{align*}
The inverse operator of $A$, denoted by $A^{-1}$, is the operator whose graph is given by $\gra A^{-1} = \left\{ (u,x)\in\Hilbert\times\Hilbert : u\in A(x)\right\}$. The identity operator is denoted by $\Id$. When $L:\Hilbert\to\Gilbert$ is a bounded linear operator, we use $L^*:\Gilbert\to\Hilbert$ to denote its \emph{adjoint}, which is the unique bounded linear operator such that $\langle Lx,y\rangle = \langle x, L^*y\rangle$, for all $x\in\Hilbert$ and $y\in\Gilbert$.

To simplify the notation, we will use $\llbracket k,l\rrbracket$ to denote the set of integers between $k,l\in\mathbb{N}$, i.e.,
\begin{equation*}
\llbracket k, l\rrbracket := \left\{
\begin{array}{ll}
\{k, k+1, \ldots ,l\} & \text{if } k\leq l, \\
\emptyset & \text{otherwise}.
\end{array}
\right.
\end{equation*}
\begin{definition}\label{def:cocoercive}
An operator $T:\Hilbert \to \Hilbert$ is said to be
\begin{enumerate}[(i)]
\item\emph{$\kappa$-Lipschitz continuous} for $\kappa >0$ if
\begin{equation*}
\|T(x)-T(y)\| \leq \kappa \|x-y\| \quad \forall x,y \in \Hilbert;
\end{equation*}
\item\label{defT:ii}\emph{nonexpansive} if it is $1$-Lipschitz continuous, i.e.,
\begin{equation*}
    \|T(x)-T(y)\| \leq \|x-y\| \quad \forall x,y \in \Hilbert;
\end{equation*}
\item\label{defT:iii}$\alpha$-\emph{averaged nonexpansive} for $\alpha\in{]0,1[}$ if
\begin{equation*}
\|T(x)-T(y)\|^2+\frac{1-\alpha}{\alpha} \|(\Id-T)(x)-(\Id-T)(y)\|^2 \leq \|x-y\|^2 \quad \forall x,y\in\Hilbert.
\end{equation*}
\end{enumerate}
\end{definition}
\begin{definition}
A set-valued operator $A:\Hilbert \setto \Hilbert$ is monotone if
\begin{equation*}
    \langle x-y,u-v\rangle \geq 0 \quad \forall (x,u),(y,v)\in\gra{A}.
\end{equation*}
Furthermore, $A$ is said to be maximally monotone if there exists no monotone operator $B:\Hilbert \setto \Hilbert$ such that $\gra{B}$ properly contains $\gra{A}$.
\end{definition}
\begin{definition}\label{def:resolvent}
Given an operator $A\colon\Hilbert\setto\Hilbert$, the \emph{resolvent} of $A$ with parameter $\gamma>0$ is the operator $J_{\gamma A}\colon\Hilbert\setto\Hilbert$ defined by $J_{\gamma A}:=(\Id+\gamma A)^{-1}$.
\end{definition}

The next result contains Minty's theorem~\cite{Minty1962}.
\begin{proposition}[{\cite[Corollary~23.11]{bauschke2017}}]
Let $A:\Hilbert\setto\Hilbert$ be monotone and let $\gamma>0$. Then,
\begin{enumerate}[(i)]
\item $J_{\gamma A}$ is single-valued,
\item $\dom J_{\gamma A}=\Hilbert$ if and only if $A$ is maximally monotone.
\end{enumerate}
\end{proposition}

\subsection{Parametrized resolvent splitting}\label{subsect:resolventsplitting}
Besides developing lifted splitting algorithms with reduced dimension, different works have been devoted to determine the minimal dimension reduction that can be achieved under some conditions. This is the case of~\cite{malitsky2021resolvent,ryu2020}, where a minimality result is obtained for the classical monotone inclusion Problem~\ref{eq:monotoneinclusion}.
In what follows, we employ $T$ for denoting a \emph{fixed point operator} and $S$ for a \emph{solution operator}, both depending on the maximally monotone operators appearing in the problem.
\begin{definition}[Fixed point encoding~\cite{ryu2020}]
A pair of operators $(T,S)$ is a \emph{fixed point encoding} for Problem~\ref{eq:monotoneinclusion} if, for all particular instance of the problem,
\begin{equation*}
\Fix{T}\neq \emptyset \Longleftrightarrow \zer{\left(\sum_{i=1}^n A_i\right)}\neq\emptyset \text{ and } \mathbf{z}\in\Fix{T} \Longrightarrow S(\mathbf{z}) \in\zer{\left(\sum_{i=1}^n A_i\right)}.
\end{equation*}
\end{definition}
Previous works on minimality are based on the concept of \emph{resolvent splitting}, which does not allow employing parametrized resolvents (i.e., it only permits computation of the resolvents $J_{A_1},\ldots,J_{A_n}$). In this work, we introduce the notion of \emph{parametrized resolvent splitting} and adapt the minimality result in~\cite[Section~3]{malitsky2021resolvent} to the more general parametrized setting. Since the reasoning is very similar to the one  in the mentioned reference, we only present the results here and refer the interested reader to Appendix~\ref{sect:appendix} for a detailed demonstration.

\begin{definition}[Parametrized resolvent splitting]
A fixed point encoding $(T,S)$ for Problem~\ref{eq:monotoneinclusion} is a \emph{parametrized resolvent splitting} if, for all particular instances of the problem, there is a finite procedure that evaluates $T$ and $S$ at a given point which only uses vector addition, scalar multiplication, and the parametrized resolvents of $A_1,\ldots,A_n$.
\end{definition}
\begin{definition}[Frugality]
A parametrized resolvent splitting $(T,S)$ for Problem~\ref{eq:monotoneinclusion} is \emph{frugal} if, in addition, each of the parametrized resolvents of $A_1,\ldots,A_n$ is used exactly once.
\end{definition}
\begin{definition}[Lifting~\cite{ryu2020}]
Let $d\in\mathbb{N}$. A fixed point encoding $(T,S)$ is a $d$-fold lifting for Problem~\ref{eq:monotoneinclusion} if $T:\Hilbert^d\to\Hilbert^d$ and $S:\Hilbert^d\to\Hilbert$.
\end{definition}
\begin{example}
In~\cite{Campoy21}, a product space reformulation with reduced dimension  is proposed, which applied to Problem~\ref{eq:monotoneinclusion} yields the following lifted splitting. Given any $\gamma > 0$ and $\lambda\in{]0,2]}$, the algorithm in~\cite[Theorem~5.1]{Campoy21} can be defined by the operator $R:\Hilbert^{n-1}\to\Hilbert^{n-1}$ given by
\begin{equation*}
R(\mathbf{z}) := \mathbf{z} + \lambda\colvec{x_1 - x_0 \\ x_2 - x_0 \\ \vdots \\ x_{n-1} - x_0} ,
\end{equation*}
where $\mathbf{z} = (z_0,z_1,\ldots,z_{n-1})$ and $\mathbf{x} = (x_0,x_1,\ldots,x_{n-1})\in\Hilbert^{n}$ is the vector defined as
\begin{equation*}
\left\{
\begin{aligned}
x_0 & =  J_{\frac{\gamma}{n-1} A_n} \left( \frac{1}{n-1} \sum_{i=1}^{n-1} z_i\right),  \\
x_i & =  J_{\gamma A_i} (2 x_0 - z_i) \quad \forall i\in{\llbracket 1,n-1\rrbracket}.
\end{aligned}
\right.
\end{equation*}
Moreover, if we let $S:\Hilbert^{n-1} \to \Hilbert$ be the operator given by
\begin{equation*}
S(\mathbf{z}) := J_{\frac{\gamma}{n-1} A_{n}} \left( \frac{1}{n-1} \sum_{i=1}^{n-1} z_i\right),
\end{equation*}
then the pair $(R,S)$ is a frugal parametrized resolvent splitting with $(n-1)$-fold lifting which is not a resolvent splitting, since it makes use of resolvent parameters.
\end{example}
Malitsky and  Tam prove in~\cite [Theorem~3.3]{malitsky2021resolvent} that the minimal lifting that one can achieve for Problem~\ref{eq:monotoneinclusion} with frugal resolvent splittings  is $n-1$. From their proof, it cannot be directly determined whether the same result holds when the resolvents are allowed to have different parameters. The next theorem provides an affirmative answer to this question.
\begin{theorem}[Minimal lifting for frugal parametrized splittings]\label{th:minimallifting}
Let $n\geq 2$ and let $(T,S)$ be a frugal parametrized resolvent splitting with $d$-fold lifting for Problem~\ref{eq:monotoneinclusion}. Then, $d \geq n-1$.
\end{theorem}
\begin{proof}
See Theorem~\ref{th:minimalliftingappendix} in Appendix~\ref{sect:appendix}.
\end{proof}


\section{A primal-dual splitting with minimal lifting }\label{sect:alg}
In this section we devise a primal-dual splitting algorithm for Problem~\ref{problem:PD} with minimal lifting.
We base our analysis in the case in which the primal problem involves only one linear composition, i.e. $m=1$, and later extend to an arbitrary finite number of  linearly composed maximally monotone operators by appealing to a product space reformulation.

Let $n\geq2$. We start by considering the primal-dual problem given by
\begin{equation}\label{eq:primal}
    \text{find } x \in \Hilbert \text{ such that } 0\in\sum_{i=1}^n A_i(x) + L^* B (Lx),
\end{equation}
and
\begin{equation}\label{eq:dual}
    \text{find } u\in\Gilbert \text{ such that } 0\in -L \left( \sum_{i=1}^n A_i \right)^{-1} \bigl(- L^* u\bigr) + B^{-1} (u),
\end{equation}
where $A_1,\ldots,A_n:\Hilbert\setto\Hilbert$  and  $B:\Gilbert \setto \Gilbert$ are maximally monotone operators  and $L:\Hilbert\to\Gilbert$ is a bounded linear operator.  Note that in this case~\eqref{eq:dual} corresponds to the Attouch--Th\'era dual problem of~\eqref{eq:primal}, see~\cite{attouch1996general}.
In the following, we denote the set of solutions of~\eqref{eq:primal} and~\eqref{eq:dual} by $\mathcal{P}$ and $\mathcal{D}$, respectively, and consider the set $\mathbf{Z}$ defined as
\begin{equation*}
    \mathbf{Z} := \left\{ (x,u)\in \Hilbert \times \Gilbert : -L^*u \in\sum_{i=1}^n A_i(x) \text{ and } u \in B(Lx)\right\},
\end{equation*}
which is useful for tackling primal-dual inclusion problems.
It is well-known that $\mathbf{Z}$ is a subset of $\Primal \times \Dual$ and that
\begin{equation*}
    \Primal \neq \emptyset  \Longleftrightarrow \mathbf{Z} \neq \emptyset \Longleftrightarrow  \Dual \neq \emptyset.
\end{equation*}
Indeed, we have
\begin{equation*}
    \begin{aligned}
        \exists\, x\in\mathcal{P} &  \Longleftrightarrow (\exists\, x\in\Hilbert)\quad 0\in\sum_{i=1}^n A_i(x)+L^*B(Lx)  \\ & \Longleftrightarrow
        (\exists\, (x,u) \in\Hilbert \times \Gilbert) \; \left\{
        \begin{aligned}
        -L^*(u) & \in \sum_{i=1}^n A_i(x) \\
        u & \in B(Lx)
        \end{aligned}
        \right.\\
          &  \Longleftrightarrow
         (\exists\, (x,u) \in\Hilbert \times \Gilbert) \; \left\{
        \begin{aligned}
        x &\in \left(\sum_{i=1}^n  A_i\right)^{-1}\bigl(-L^*u\bigr) \\
        Lx  &\in B^{-1} (u)
        \end{aligned}
        \right.  \\ &\Longleftrightarrow
        (\exists\, u\in\Gilbert) \quad 0\in -L \left( \sum_{i=1}^n A_i \right)^{-1} \bigl(- L^* u\bigr) + B^{-1} (u) \Longleftrightarrow \exists\, u \in\mathcal{D}.
    \end{aligned}
\end{equation*}
We refer to an element of $\mathbf{Z}$ as a \emph{primal-dual solution} of~\eqref{eq:primal}-\eqref{eq:dual}.

Now, we introduce a fixed point algorithm for solving the primal-dual problem given by~\eqref{eq:primal}-\eqref{eq:dual}. Let $\lambda,\gamma >0$ and let $T:\Hilbert^{n-1} \times \Gilbert \to \Hilbert^{n-1} \times \Gilbert$ be the operator given by
\begin{equation}\label{eq:defT}
    T\colvec{\mathbf{z}\\v}: = \colvec{\mathbf{z}\\v} + \lambda \colvec{x_2 - x_1\\ x_3- x_2\\ \vdots\\ x_n-x_{n-1}\\ \gamma(y-Lx_n)},
\end{equation}
where $(\mathbf{x},y)=(x_1,\ldots,x_n,y) \in \Hilbert^n\times \Gilbert$ depends on $(\mathbf{z},v) = (z_1,\ldots,z_{n-1},v) \in\Hilbert^{n-1} \times \Gilbert$ in the following way
\begin{equation}\label{eq:xy}
    \left\{
    \begin{aligned}
        x_1 &= J_{A_1}(z_1), \\
        x_i &= J_{A_i}(z_i+x_{i-1}-z_{i-1}), \quad \forall i\in{\llbracket 2,n-1\rrbracket},  \\
        x_n &=  J_{A_n}(x_1+x_{n-1} -z_{n-1}- L^* (\gamma Lx_1 - v)), \\
        y & =  J_{B/\gamma} \left(L(x_1 + x_{n})- \frac{v}{\gamma}\right).
    \end{aligned}
    \right.
\end{equation}

In the next lemma we characterize the set of fixed points of the operator $T$ by means of the set of primal-dual solutions to~\eqref{eq:primal}-\eqref{eq:dual}.\newpage

\begin{lemma}\label{lema:ZFix} Let $n\geq 2$ and $\lambda, \,\gamma > 0$. The following assertions hold.
\begin{enumerate}[(i)]
    \item\label{it:ZFix-i} If $(\bar{x},\bar{u})\in\mathbf{Z}$, then there exists $\barbf{z}\in\Hilbert^{n-1}$  such that  $(\barbf{z},\gamma L\bar{x}-\bar{u})\in\Fix{T}$.
    \item\label{it:ZFix-ii} If $(\bar{z}_1,\ldots,\bar{z}_{n-1},\bar{v})\in\Fix{T}$, then $(J_{A_1}(\bar{z}_1),\gamma L\bar{x}-\bar{v})\in\mathbf{Z}$.
\end{enumerate}
As a result,
\begin{equation*}
    \Fix{T} \neq \emptyset \Longleftrightarrow \mathbf{Z} \neq \emptyset.
\end{equation*}
\end{lemma}
\begin{proof}
(\ref{it:ZFix-i}) Let $(\bar{x},\bar{u})\in\mathbf{Z}$. Then, $\bar{u} \in B(L\bar{x})$ and there exists $(a_1,\ldots,a_n)\in\Hilbert^{n}$ such that $a_i\in A_i(\bar{x})$ and $-L^*\bar{u} = \sum_{i=1}^n a_i$. Consider the vectors $(\bar{z}_1,\ldots,\bar{z}_{n-1},\bar{v})\in \Hilbert^{n-1}\times \Gilbert$ defined as
\begin{equation*}
\left\{
\begin{aligned}
\bar{z}_1 & := \bar{x} + a_1 \in (\Id + A_1)(\bar{x}), \\
\bar{z}_i & := a_i + \bar{z}_{i-1} = (\Id + A_i)(\bar{x}) - \bar{x} + \bar{z}_{i-1}, \quad \forall i\in{\llbracket2,n-1\rrbracket}, \\
\bar{v} & : = \gamma L\bar{x}-\bar{u} \in \left(\gamma \Id - B\right)(L\bar{x}).
\end{aligned}
\right.
\end{equation*}
Then, we deduce that $\bar{x} = J_{A_1}(\bar{z}_1)$ and $\bar{x} = J_{A_i} ( \bar{z}_i + \bar{x}-\bar{z}_{i-1})$ for all $i\in{\llbracket 2,n-1\rrbracket}$. Moreover, we have
\begin{equation*}
    \begin{aligned}
    2\bar{x} - \bar{z}_{n-1}- L^*(\gamma L\bar{x}-\bar{v})& = 2\bar{x} -\bar{z}_{n-1} - L^*(\bar{u}) \\
    & = \bar{x} + a_n + \bar{x}-\bar{z}_{n-1} + \sum_{i=1}^{n-1} a_i \\
    & = \bar{x}+a_n + \bar{x}-\bar{z}_{n-1} + \sum_{i=2}^{n-1} (\bar{z}_i - \bar{z}_{i-1}) + \bar{z}_1-\bar{x}  = (\Id + A_n)(\bar{x}).
    \end{aligned}
\end{equation*}
Altogether, we obtain
\begin{equation*}
    \left\{
    \begin{aligned}
        \bar{x} &= J_{A_1}(\bar{z}_1), \\
        \bar{x} &= J_{A_i}(\bar{z}_i+\bar{x}-\bar{z}_{i-1}), \quad \forall i\in{\llbracket 2,n-1\rrbracket},  \\
        \bar{x} &=  J_{A_n}(2\bar{x} -\bar{z}_{n-1}- L^* (\gamma L\bar{x} - \bar{v})),  \\
        L\bar{x} & =  J_{B/\gamma}\left (2L\bar{x}- \frac{\bar{v}}{\gamma}\right),
    \end{aligned}
    \right.
\end{equation*}
which implies that $(\bar{z}_1,\ldots,\bar{z}_{n-1},\bar{v})\in\Fix{T}$.

(\ref{it:ZFix-ii}) Let $(\bar{z}_1,\ldots,\bar{z}_{n-1},\bar{v})\in\Fix{T}$ and set $\bar{x}:= J_{A}(\bar{z}_1)$. By~\eqref{eq:defT}, $y = L\bar{x}$ and  $\bar{x}_i = \bar{x}$ for all $i=1,\ldots,n$. Consequently, from~\eqref{eq:xy} we derive
\begin{equation*}
\left\{
    \begin{array}{l}
    \bar{z}_1-\bar{x} \in A_1(\bar{x}), \\
    \bar{z}_i-\bar{z}_{i-1} \in A_i(\bar{x}), \quad \forall i\in{\llbracket 2,n-1\rrbracket}, \\
    \bar{x}-\bar{z}_{n-1}-L^*(\gamma L\bar{x}-\bar{v})\in A_n(\bar{x}),  \\
    \gamma L\bar{x}-\bar{v} \in B (L\bar{x}).
    \end{array}
\right.
\end{equation*}
Summing together the first $n$ inclusions above and setting $\bar{u}:=\gamma L\bar{x}-\bar{v}$, we deduce
\begin{equation*}
\left\{
    \begin{aligned}
    -L^*\bar{u} & \in \sum_{i=1}^n A_i(\bar{x}), \\
    \bar{u} & \in B(L\bar{x}),
    \end{aligned}
    \right.
\end{equation*}
which implies $(\bar{x},\bar{u})\in\mathbf{Z}$, as claimed.
\end{proof}
The following lemma provides  nonexpansive properties of the operator $T$ in the Hilbert space $\Hilbert^{n-1} \times \Gilbert $ with scalar product given by
\begin{equation}\label{eq:normgamma}
    \langle (z_1,\ldots,z_{n-1},v),(\bar{z}_1,\ldots,\bar{z}_{n-1},\bar{v}) \rangle_{\gamma} : = \sum_{i=1}^{n-1} \langle z_i, \bar{z}_i\rangle_{\Hilbert} + \frac{1}{\gamma} \langle v, \bar{v} \rangle_{\Gilbert},
\end{equation}
for $(z_1,\ldots,z_{n-1},v),(\bar{z}_1,\ldots,\bar{z}_{n-1},\bar{v})\in\Hilbert^{n-1}\times\Gilbert$ and $\gamma > 0$.
\begin{lemma}\label{lema:averaged}
For all $(\mathbf{z},v) = (z_1,\ldots,z_{n-1},v)\in\Hilbert^{n-1}\times\Gilbert$ and  $(\barbf{z},\bar{v}) = (\bar{z}_1,\ldots,\bar{z}_{n-1},\bar{v})\in\Hilbert^{n-1}\times\Gilbert$,
\begin{equation}\label{eq:averaged}
\begin{aligned}
    \| T(\mathbf{z},v)&-T(\barbf{z},\bar{v})\|^2 _{\gamma} + \frac{1-\lambda}{\lambda}  \| \left(\Id-T\right)(\mathbf{z},v)-\left(\Id-T\right)(\barbf{z},\bar{v})\|^{2}_{\gamma}\\
    &+ \frac{1-\gamma \|L\|^2}{\lambda} \left\| \sum_{i=1}^{n-1} \left(\Id-T\right)(\mathbf{z},v)_{i} - \sum_{i=1}^{n-1} \left(\Id-T\right)(\barbf{z},\bar{v})_{i}\right\|^2_{\gamma} \leq \|(\mathbf{z},v)-(\barbf{z},\bar{v})\|^2_{\gamma},
\end{aligned}
\end{equation}
where $\|\cdot\|_\gamma$ denotes the norm induced by the scalar product~\eqref{eq:normgamma}.
In particular, if $\lambda \in{]0,1[}$ and $\gamma\in{\left]0,\frac{1}{\|L\|^2}\right]}$, the operator $T$ is $\lambda$-averaged nonexpansive.
\end{lemma}
\begin{proof}
Let $(x_1,\ldots,x_{n},y)\in\Hilbert^{n}\times\Gilbert$ and $(\bar{x}_1,\ldots,\bar{x}_n,\bar{y})\in\Hilbert^n \times \Gilbert$ be given by~\eqref{eq:xy} from $(\mathbf{z},v)$ and $(\barbf{z},\bar{v})$, respectively. For simplicity, we denote $(\mathbf{z}^{+},v^{+})= T(\mathbf{z},v)$ and $(\barbf{z}^{+},\bar{v}^{+})=T(\barbf{z},\bar{v})$.  Since $z_1-x_1\in A_1(x_1)$ and $\bar{z}_1-\bar{x}_1\in A_1(\bar{x}_1)$, by monotonicity of $A_1$
\begin{equation}\label{eq:monoA1}
\begin{aligned}
    0 & \leq \langle  (z_1-x_1) - (\bar{z}_1-\bar{x}_1), x_1-\bar{x}_1 \rangle \\
    & = \langle (z_1-x_1) -(\bar{z}_1-\bar{x}_1), x_1-x_2\rangle + \langle  (z_1-x_1) -(\bar{z}_1-\bar{x}_1), x_2 - \bar{x}_1\rangle.
    \end{aligned}
\end{equation}
For every $i\in{\llbracket 2,n-1\rrbracket}$, we have $z_i+x_{i-1}-z_{i-1}-x_i \in A_i(x_i)$ and $\bar{z}_i + \bar{x}_{i-1}-\bar{z}_{i-1}-\bar{x}_i\in A_i(\bar{x}_i)$ and thus, by monotonicity of $ A_i$
\begin{equation}\label{eq:monoAi}
    \begin{aligned}
    0 & \leq \langle (z_i+x_{i-1}-z_{i-1}-x_i)- (\bar{z}_i + \bar{x}_{i-1}-\bar{z}_{i-1}-\bar{x}_i), x_i - \bar{x}_i \rangle \\
    & = \langle (z_i-x_i)-(\bar{z}_i-\bar{x}_i), x_i-\bar{x}_i\rangle - \langle (z_{i-1}-x_{i-1})-(\bar{z}_{i-1}-\bar{x}_{i-1}), x_i-\bar{x}_i\rangle \\
    & = \langle (z_i-x_i)-(\bar{z}_i-\bar{x}_i),x_i-x_{i+1}\rangle + \langle (z_i-x_i)-(\bar{z}_i-\bar{x}_i), x_{i+1}-\bar{x}_i\rangle \\
    & \quad - \langle (z_{i-1}-x_{i-1})-(\bar{z}_{i-1}-\bar{x}_{i-1}), x_i-\bar{x}_{i-1} \rangle \\
    & \quad- \langle (z_{i-1}-x_{i-1})-(\bar{z}_{i-1}-\bar{x}_{i-1}),\bar{x}_{i-1}-\bar{x}_i\rangle.
    \end{aligned}
\end{equation}
Now, since  $x_1+x_{n-1}-z_{n-1}-x_n - L^*\left( \gamma Lx_1-v\right) \in A_n(x_n) $ and $\bar{x}_1+\bar{x}_{n-1}-\bar{z}_{n-1}-\bar{x}_n - L^*\left(\gamma L\bar{x}_1-\bar{v}\right)\in A_n(\bar{x}_n)$, again monotonicity of $A_n$ results in the inequality
\begin{equation}\label{eq:monoAn}
\begin{aligned}
    0 & \leq \left\langle x_1+x_{n-1}-z_{n-1}-x_n - L^*\left( \gamma Lx_1-v\right), x_n-\bar{x}_n\right\rangle\\
    & \quad
    -\left\langle \bar{x}_1+\bar{x}_{n-1}-\bar{z}_{n-1}-\bar{x}_n - L^*\left( \gamma L\bar{x}_1-\bar{v}\right), x_n-\bar{x}_n\right\rangle\\
    & = \langle (x_{n-1}-z_{n-1}) -(\bar{x}_{n-1}-\bar{z}_{n-1}), x_{n}-\bar{x}_n \rangle + \langle (x_1-\bar{x}_1)-(x_{n}-\bar{x}_{n}),x_n-\bar{x}_{n} \rangle \\
    & \quad - \langle \gamma \left(Lx_1-L\bar{x}_1\right)-(v-\bar{v}), Lx_n-L\bar{x}_n\rangle \\
    & = \langle (x_{n-1}-z_{n-1}) -(\bar{x}_{n-1}-\bar{z}_{n-1}), x_{n}- \bar{x}_{n-1} \rangle+ \langle (x_1-\bar{x}_1)-(x_{n}-\bar{x}_{n}),x_n-\bar{x}_{n} \rangle\\
    & \quad + \langle (x_{n-1}-z_{n-1}) -(\bar{x}_{n-1}-\bar{z}_{n-1}), \bar{x}_{n-1}-  \bar{x}_n \rangle \\
    &\quad - \langle \gamma \left(Lx_1-L\bar{x}_1\right)-(v-\bar{v}), Lx_n-L\bar{x}_n\rangle. \\
\end{aligned}
\end{equation}
Finally, we have $\gamma L(x_1+x_n)-v -\gamma y \in B(y)$ and $\gamma L(\bar{x}_1+\bar{x}_n)-\bar{v}-\gamma \bar{y}\in B(\bar{y})$, so by monotonicity of $B$ we get
\begin{equation}\label{eq:monoB}
    \begin{aligned}
    0 & \leq \langle (\gamma L(x_1+x_n) -v-\gamma \,y)-(\gamma L(\bar{x}_1+\bar{x}_n)-\bar{v}-\gamma\, \bar{y}), y-\bar{y} \rangle.  \\
    \end{aligned}
\end{equation}
Summing together~\eqref{eq:monoA1}-\eqref{eq:monoB} and rearranging, yields
\begin{equation}\label{eq:monoAB}
    \begin{aligned}
    0 & \leq \sum_{i=1}^{n-1} \langle (x_i-x_{i+1})- (\bar{x}_i-\bar{x}_{i+1}), z_i-\bar{z}_i\rangle  + \sum_{i=1}^{n-1} \langle (x_i-\bar{x}_i)-(x_{i+1}-\bar{x}_{i+1}), \bar{x}_i-x_i\rangle \\
    & \quad + \langle (x_1-\bar{x}_1) - (x_n-\bar{x}_n), x_n-\bar{x}_n\rangle +\langle \left(Lx_n-L\bar{x}_n\right)-(y-\bar{y}), v-\bar{v}\rangle  \\
     & \quad + \gamma \langle \bigl(L(x_1+x_n)-L(\bar{x}_1+\bar{x}_n)\bigr)- (y-\bar{y}), y-\bar{y}\rangle \\
    & \quad - \gamma \langle Lx_1-L\bar{x}_1,Lx_n-L\bar{x}_n\rangle .
    \end{aligned}
\end{equation}
The sums in~\eqref{eq:monoAB} can be written, respectively, as
\begin{equation}\label{eq:sum1}
    \begin{aligned}
    \sum_{i=1}^{n-1}& \langle (x_i-x_{i+1})-(\bar{x}_i-\bar{x}_{i+1}), z_i-\bar{z}_i\rangle  \\
     & = \frac{1}{\lambda} \sum_{i=1}^{n-1}\langle  (z_i-z_i^{+})-(\bar{z}_i-\bar{z}_i^{+}), z_i-\bar{z}_i\rangle \\
     & = \frac{1}{\lambda} \langle (\mathbf{z}-\mathbf{z}^{+})-(\barbf{z}-\barbf{z}^{+}),\mathbf{z}-\barbf{z}\rangle \\
     & = \frac{1}{2\lambda} \left( \|(\mathbf{z}-\mathbf{z}^{+})-(\barbf{z}-\barbf{z}^{+})\|^2 - \| \mathbf{z}^{+}-\barbf{z}^{+}\|^2 +  \|\mathbf{z}- \barbf{z}\|^2 \right),
    \end{aligned}
\end{equation}
and
\begin{equation}\label{eq:sum2}
\begin{aligned}
 \sum_{i=1}^{n-1}& \langle (x_i-\bar{x}_i)-(x_{i+1}-\bar{x}_{i+1}), \bar{x}_i-x_i\rangle \\
 & = \frac{1}{2} \sum_{i=1}^{n-1} \left( \|x_{i+1}-\bar{x}_{i+1}\|^2 - \|x_i-\bar{x}_i\|^2 - \| (x_i-x_{i+1})-(\bar{x}_i-\bar{x}_{i+1})\|^2 \right) \\
 & = \frac{1}{2} \left( \|x_n-\bar{x}_n\|^2 - \|x_1 -\bar{x}_1\|^2 - \frac{1}{\lambda^2} \sum_{i=1}^{n-1} \| (z_i-z_i^{+})-(\bar{z}_i- \bar{z}_i^{+})\|^2 \right) \\
 & = \frac{1}{2} \left( \|x_n-\bar{x}_n\|^2 - \|x_1 -\bar{x}_1\|^2 - \frac{1}{\lambda ^2} \| (\mathbf{z}-\mathbf{z}^{+}) - (\barbf{z}-\barbf{z}^{+})\|^2 \right).
\end{aligned}
\end{equation}
The third term in~\eqref{eq:monoAB}, becomes
\begin{equation}\label{eq:term3}
    \langle (x_1-\bar{x}_1) - (x_n-\bar{x}_n), x_n-\bar{x}_n\rangle = \frac{1}{2} \left( \|x_1-\bar{x}_1\|^2 - \| x_n-\bar{x}_n\|^2 - \|(x_1-\bar{x}_1)-(x_n-\bar{x}_n)\|^2 \right),
\end{equation}
while the fourth term yields
\begin{equation}\label{eq:term4}
\begin{aligned}
\langle (Lx_n-&L\bar{x}_n) - (y-\bar{y}), v-\bar{v} \rangle   \\
& = \frac{1}{\gamma \lambda } \langle (v-v^{+})-(\bar{v}-\bar{v}^{+}), v-\bar{v}\rangle  \\
& = \frac{1}{2\gamma \lambda} \left( \|(v-v^{+}) -(\bar{v}-\bar{v}^{+})\|^2 - \|v^{+}-\bar{v}^{+}\|^2 + \|v-\bar{v}\|^2 \right).
\end{aligned}
\end{equation}
Lastly, making use of the Cauchy--Schwarz and Young's inequalities, the second last term of~\eqref{eq:monoAB} gives
\begin{equation}\label{eq:term6}
    \begin{aligned}
    \gamma \langle &\bigl(L(x_1+x_n)-L(\bar{x}_1+\bar{x}_n)\bigr)- (y-\bar{y}), y-\bar{y}\bigr \rangle  \\
    & = \gamma \left(\langle  Lx_1-L\bar{x}_1, y-\bar{y}\rangle + \langle  \left(Lx_n-L\bar{x}_n\right)-(y-\bar{y}), y-\bar{y}\rangle \right) \\
    & = \frac{\gamma}{2} \left(\|Lx_n-L\bar{x}_n\|^2 -  \|\left(Lx_n-L\bar{x}_n\right)-(y-\bar{y})\|^2 -\|y-\bar{y}\|^2 \right) \\
    &\quad + \gamma \langle  Lx_1-L\bar{x}_1, y-\bar{y}\rangle  \\
    & \leq \frac{\gamma}{2} \left(\|Lx_n-L\bar{x}_n\|^2 - \frac{1}{\gamma ^2 \lambda^2} \|(v-v^{+})-(\bar{v}- \bar{v}^{+})\|^2 -\|y-\bar{y}\|^2 \right)\\
    & \quad + \frac{\gamma}{2} \|Lx_1-L\bar{x}_1\|^2 + \frac{\gamma}{2} \|y-\bar{y}\|^2 \\
    & = \frac{\gamma}{2} \|Lx_1-L\bar{x}_1\|^2 + \frac{\gamma}{2} \|Lx_n-L\bar{x}_n\|^2 - \frac{1}{2 \gamma \lambda^2} \|(v-v^{+})-(\bar{v}- \bar{v}^{+})\|^2,
    \end{aligned}
\end{equation}
while the last term can be rearranged as follows
\begin{equation}\label{eq:term5}
    \begin{aligned}
    - \gamma \langle &Lx_1-L\bar{x}_1, Lx_n-L\bar{x}_n\rangle  \\
    & = \frac{\gamma }{2} \left( \| L(x_1-x_n) - L(\bar{x}_1-\bar{x}_n)\|^2 - \|Lx_1-L\bar{x}_1\|^2 - \|Lx_n-L\bar{x}_n\|^2 \right).
    \end{aligned}
\end{equation}
Summing together~\eqref{eq:term6} and~\eqref{eq:term5} and using the Lipschitz continuity of $L$, we
get
\begin{equation}\label{eq:last2}
    \begin{aligned}
     \gamma \langle& \bigl(L(x_1+x_n)-L(\bar{x}_1+\bar{x}_n)\bigr)- (y-\bar{y}), y-\bar{y}\rangle - \gamma \langle Lx_1-L\bar{x}_1,Lx_n-L\bar{x}_n\rangle \\
    & = \frac{\gamma }{2}  \| L(x_1-x_n) - L(\bar{x}_1-\bar{x}_n)\|^2  - \frac{1}{2 \gamma \lambda^2} \|(v-v^{+})-(\bar{v}- \bar{v}^{+})\|^2 \\
    & \leq \frac{\gamma \|L\|^2 }{2}  \| (x_1-x_n) - (\bar{x}_1-\bar{x}_n)\|^2  - \frac{1}{2 \gamma \lambda^2} \|(v-v^{+})-(\bar{v}- \bar{v}^{+})\|^2.
    \end{aligned}
\end{equation}
Multiplying~\eqref{eq:monoAB} by $2\lambda$ and substituting equations~\eqref{eq:sum1}-\eqref{eq:last2}, we obtain the final inequality
\begin{equation*}
\begin{aligned}
    &\|\mathbf{z}^{+}-\barbf{z}^{+}\|^2  + \left(\frac{1}{\lambda}-1\right)\left(\|(\mathbf{z}-\mathbf{z}^{+})-(\barbf{z}-\barbf{z}^{+})\|^2 + \frac{1}{\gamma}\|(v-v^{+})-(\bar{v}-\bar{v}^{+})\|^2 \right) \\
    &  + \frac{1}{\gamma}\|v^{+}-\bar{v}^{+}\|^2+\lambda \left(1-\gamma \|L\|^2\right) \|(x_1-x_n) -(\bar{x}_1-\bar{x}_n)\|^2 \\
    & \quad \leq \|\mathbf{z}-\barbf{z}\|^2 +\frac{1}{\gamma} \|v -\bar{v}\|^2.
    \end{aligned}
\end{equation*}
To complete the proof, just note that
\begin{align*}
    \lambda  (x_1-x_n)-\lambda(\bar{x}_1-\bar{x}_n)&= \lambda \sum_{i=1}^{n-1} (x_i-x_{i+1}) - \lambda \sum_{i=1}^{n-1} (\bar{x}_i-\bar{x}_{i+1}) \\
    &= \sum_{i=1}^{n-1} (z_i-z_i^{+}) - \sum_{i=1}^{n-1} (\bar{z}_i-\bar{z}_i^{+}),
\end{align*}
from where~\eqref{eq:averaged}  finally follows.
\end{proof}

Next we state our main result, which establishes the convergence of the iterative algorithm defined by the operator $T$ in~\eqref{eq:defT}-\eqref{eq:xy}.\pagebreak
\begin{theorem}\label{t:conv}
Let $n\geq 2$, let $L:\Hilbert \to \Gilbert$ be a bounded linear operator and let $A_1,\ldots,A_n:\Hilbert\setto\Hilbert$ and $B:\Gilbert\setto\Gilbert$ be maximally monotone operators with $\zer{\left(\sum_{i=1}^n A_i + L^*BL\right)\neq\emptyset}$. Further, let $\lambda\in{]0,1[}$ and $\gamma\in{\left]0,\frac{1}{\|L\|^2}\right]}$. Given an initial point $(\mathbf{z}^0,v^0) = (z_1^0,\ldots,z_{n-1}^0,v^0)\in\Hilbert^{n-1}\times\Gilbert$, consider the sequences given by
\begin{equation}\label{eq:zkvk}
    \colvec{\mathbf{z}^{k+1}\\v^{k+1}} = \colvec{\mathbf{z}^k\\v^k} + \lambda \colvec{x_2^k - x_1^k\\ x_3^k- x_2^k\\ \vdots\\ x_n^k-x_{n-1}^k\\ \gamma(y^k-Lx_n^k)} \quad  \forall k\geq 0,
\end{equation}
with
\begin{equation}\label{eq:xkyk}
    \left\{
    \begin{aligned}
        x_1^k &= J_{A_1}(z_1^k), \\
        x_i^k &= J_{A_i}(z_i^k+x_{i-1}^k-z_{i-1}^k), \quad \forall i\in{\llbracket 2,n-1\rrbracket},  \\
        x_n^k &=  J_{A_n}(x_1^k+x_{n-1}^k -z_{n-1}^k- L^* (\gamma Lx_1^k - v^k)), \\
        y^k & =  J_{B/\gamma} \left(L(x_1^k + x_{n}^k)- \frac{v^k}{\gamma}\right).
    \end{aligned}
    \right.
\end{equation}
Then the following statements hold.
\begin{enumerate}[(i)]
\item\label{it:conv-i} The sequence $(\mathbf{z}^k,v^k)_{k\in\mathbb{N}}$ converges weakly  to a point $(\barbf{z},\bar{v})\in\Fix{T}$.
\item\label{it:conv-ii} The sequence $(x^k_1,\ldots,x^k_n,y^k)_{k\in\mathbb{N}}$ converges weakly to $(\bar{x},\ldots,\bar{x},L\bar{x})$ with $\bar{x} \in \Primal$.
\item\label{it:conv-iii} The sequence $\bigl(\gamma Lx_i^k-v^k \bigr)_{k\in\mathbb{N}}$ converges weakly to $\gamma L\bar{x}-\bar{v}\in\Dual$, for all $i\in{\llbracket 1,n\rrbracket}$.
\end{enumerate}
\end{theorem}
\begin{proof}
(\ref{it:conv-i}) The sequence in~\eqref{eq:zkvk} is the fixed point iteration generated as
\begin{equation*}
\colvec{\mathbf{z}^{k+1}\\v^{k+1}} = T \colvec{\mathbf{z}^k\\v^k} \quad \forall k\geq 0.
\end{equation*}
Since $\lambda\in{]0,1[}$ and $\gamma\in{\left]0,\|L\|^{-2}\right]}$,  $T$ is averaged nonexpansive by Lemma~\ref{lema:averaged} and, moreover, $\Fix{T}=\emptyset$, due to $\mathbf{Z} \neq \emptyset$ and Lemma~\ref{lema:ZFix}(\ref{it:ZFix-i}). Then, by~\cite[Theorem~5.15]{bauschke2017} the sequence $(\mathbf{z}^k,v^k)_{k\in\mathbb{N}}$ converges weakly to a point $(\barbf{z},\bar{v})\in\Fix{T}$ and $\lim_{k\to\infty} \|(\mathbf{z}^{k+1},v^{k+1})-(\mathbf{z}^k,v^k)\|_{\gamma}= 0$.

(\ref{it:conv-ii}) From~(\ref{it:conv-i}), the sequence $(\mathbf{z}^k,v^k)_{k\in\mathbb{N}}$  is bounded. Then, nonexpansivity of the resolvents and boundedness of the linear operator $L$ imply that the sequence $(\mathbf{x}^k,y^k)_{k\in\mathbb{N}}=(x_1^k,\ldots,x_n^k,y^k)_{k\in\mathbb{N}}$ is also bounded. Further, the fact that $(\mathbf{z}^{k+1},v^{k+1})_{k\in\mathbb{N}}-(\mathbf{z}^k,v^k)_{k\in\mathbb{N}}\to 0$, as $k\to\infty$, implies by~\eqref{eq:zkvk} that
\begin{equation}\label{eq:strongconv}
y^k - Lx_n^k \to 0 \text{ and } x_{i+1}^k -x_i^k \to 0, \text{ for all } i\in{\llbracket 1,n-2\rrbracket}.
\end{equation}
Next, by making use of the definition of resolvents and~\eqref{eq:xkyk}, we can write
\begin{equation}\label{eq:demiclosed2}
C\begin{pmatrix}
z_1^k-x_1^k \\
(z_2^k-x_2^k)-(z_{1}^k-x_{1}^k)  \\
\vdots \\
(z_{n-1}^k-x_{n-1}^k)-(z_{n-2}^k-x_{n-2}^k) \\
x_n^k \\
\gamma \left(L(x_1^k+x_n^k)-y^k\right)-v^k
\end{pmatrix}
 \ni
\begin{pmatrix}
x_1^k-x_n^k \\
x_2^k-x_n^k \\
\vdots \\
x_{n-1}^k-x_n^k\\
x_1^k-x_n^k +\gamma L^*\left(Lx_n^k-y^k\right)\\
y^k- Lx_n^k
\end{pmatrix},\small
\end{equation}
where the operator $C\colon\Hilbert^n\times\Gilbert\setto\Hilbert^n\times\Gilbert$ is given by
\begin{equation}\label{eq:operator s}
 C:= \begin{pmatrix}
A_1^{-1}\\A_2^{-1} \\ \vdots \\ A_{n-1}^{-1} \\ A_n\\ B^{-1}
\end{pmatrix} + \begin{pmatrix}
0 & 0 & \dots & 0 & -\Id &0 \\
0 & 0 & \dots & 0 & -\Id &0\\
\vdots & \vdots & \ddots & \vdots & \vdots &\vdots\\
0 & 0 & \dots & 0 & -\Id &0\\
\Id & \Id & \dots & \Id & 0 & L^*\\
0 & 0 &\dots &0 &- L &0
\end{pmatrix}.
\end{equation}
The operator $C$ is  maximally monotone as the sum of a maximally monotone operator and a skew symmetric linear operator (see, e.g.,~\cite[Corollary~25.5~(i) \& Example~20.35]{bauschke2017}). Thus, the graph of $C$ is sequentially closed in the weak-strong topology, by demiclosedness of maximally monotone operators~\cite[Corollary~20.38]{bauschke2017}.

Now,  let $(\barbf{x},\bar{y})$ be a weak sequential cluster point of $(\mathbf{x}^k,y^k)_{k\in\mathbb{N}}$. Due to~\eqref{eq:strongconv}, $\barbf{x}$ is of the form $\barbf{x} = (\bar{x},\ldots,\bar{x})\in\Hilbert^{n}$ and $\bar{y}=L\bar{x}$. Taking the limit along a subsequence of $(\mathbf{x}^k,y^k)_{k\in\mathbb{N}}$ which converges weakly to $(\barbf{x},\bar{y})$ and using demiclosedness of $C$, equations~\eqref{eq:demiclosed2} and~\eqref{eq:operator s} yield the expression
\begin{equation*}
\left\{
    \begin{array}{l}
    \bar{z}_1-\bar{x} \in A_1(\bar{x}),  \\
    \bar{z}_i-\bar{z}_{i-1} \in A_i(\bar{x}), \quad \forall i\in{\llbracket 2,n-1\rrbracket}, \\
    \bar{x}-\bar{z}_{n-1}-L^*(\gamma L \bar{x}-\bar{v})\in A_n(\bar{x}) , \\
    \gamma L\bar{x}-\bar{v}  \in B (L\bar{x}), 
    \end{array}
\right.
\end{equation*}
which, by summing the first $n$ equations, implies that $(\bar{x},\gamma L\bar{x}-\bar{v})\in\mathbf{Z}$ with $\bar{x} = J_{A_1}(\bar{z}_1)$. In particular, we have shown that $(\barbf{x},\bar{y})$  is directly obtained from $\barbf{z}$, implying that it is the unique weak sequential cluster point of the bounded sequence $(\mathbf{x}^k,y^k)_{k\in\mathbb{N}}$. Thus, the full sequence converges weakly to this point.

(\ref{it:conv-iii}) From~\eqref{it:conv-i}-\eqref{it:conv-ii}, for all $i\in{\llbracket 1,n\rrbracket}$, we deduce that the sequence $(\gamma Lx_i^k-v^k)_{k\in\mathbb{N}}$ weakly converges to $\gamma L\bar{x}-\bar{v}$, which belongs to $\Dual$ since $(\bar{x},\gamma L\bar{x}-\bar{v})\in\mathbf{Z}$.
\end{proof}
\begin{remark}[Malitsky--Tam resolvent splitting~\cite{malitsky2021resolvent} as a special case]
Consider Problem~\eqref{eq:primal}-\eqref{eq:dual} in the particular case in which $L=\Id$. Then, $B:\Hilbert\setto\Hilbert$ and equation~\eqref{eq:primal} becomes the classical monotone inclusion problem with $(n+1)$-operators. Furthermore, by setting $\gamma=1$ in Theorem~\ref{t:conv}, it is straightforward to see that the sequences in~\eqref{eq:zkvk}-\eqref{eq:xkyk} yield the Malitsky--Tam resolvent splitting with minimal lifting for $(n+1)$-operators.
\end{remark}
\begin{remark}[On the parameter $\gamma$ in the definition of the norm $\| \cdot\|_{\gamma}$]\label{remark:gamma}
In Lemma~\ref{lema:averaged}, we proved that the operator $T$ is $\lambda$-averaged with respect to the norm $\| \cdot \|_{\gamma}$ induced by the scalar product defined in~\eqref{eq:normgamma}. Although the use of this norm did not require detours from the usual procedure to prove convergence of the fixed point algorithm in Theorem~\ref{t:conv}, it may numerically affect the performance of the algorithm. To give an intuition about this, consider the norm of the \emph{sequence of residuals} $\left(\|(\mathbf{z}^{k+1},v^{k+1})-(\mathbf{z}^k,v^k)\|_{\gamma}\right)_{k\in\mathbb{N}}$, which converges to 0 as the algorithm reaches a fixed point, and note that we have
\begin{equation*}
\left\|(\mathbf{z}^{k+1},v^{k+1})-(\mathbf{z}^k,v^k)\right\|^2_{\gamma}  = \|\mathbf{z}^{k+1}-\mathbf{z}^k\|^2 + \frac{1}{\gamma} \|v^{k+1}-v^k\|^2 \quad \forall k\geq 0.
\end{equation*}
Lemma~\ref{lema:averaged} implies that this sequence is monotone decreasing, but if $\gamma$ is very small, the weight of the sequence of dual variables $(v^{k+1}- v^k)_{k\in\mathbb{N}}$ in the norm would be much larger than the one of the sequence of primal variables $(\mathbf{z}^{k+1}-\mathbf{z}^k)_{k\in\mathbb{N}}$, so a small decrease in the value of $\|v^{k+1}- v^k\|$  will readily imply a decrease of the norm of the sequence of residuals even if $\|\mathbf{z}^{k+1}-\mathbf{z}^k\|$ does not diminish much. Because of that, a larger number of iterations might be needed to achieve convergence of the primal sequence, which can slow down the overall convergence of the algorithm. Nonetheless, it is possible to perform some sort of pre-conditioning to prevent from having a large constant in the definition of the norm. We will further comment on this in the numerical experiment in Section~\ref{sect:numerics}.
\end{remark}
A standard product space reformulation permits to extend our method to the more general inclusion Problem~\ref{problem:PD}, which has finitely many linearly composed maximally monotone operators. We detail this in the following corollary, while the resulting scheme is displayed in Algorithm~\ref{alg:liftedPD}.
\begin{algorithm}[!ht]
\caption{Primal-dual splitting for Problem~\ref{problem:PD} with $(n-1,m)$-lifting, with $n\geq 2$.}\label{alg:liftedPD}
\begin{algorithmic}[1]
\Require{$\lambda\in{]0,1[}$ and $\gamma\in{\left]0,1/\sum_{j=1}^m \|L_j\|^2\right]}$.}
\State{Choose $\mathbf{z}^0 = (z_1^0,\ldots,z_{n-1}^0)\in\Hilbert^{n-1}$ and $\mathbf{v}^0 = (v_1^0,\ldots,v_m^0)\in\Gilbert_1\times\cdots\times\Gilbert_m$.}
\For{$k=0,1,\dots$}
\State{Compute 
\begin{equation}\label{eq:zkvkalg}
  \colvec{\mathbf{z}^{k+1} \\ \mathbf{v}^{k+1}} = \colvec{\mathbf{z}^{k} \\ \mathbf{v}^{k}} + \lambda
\colvec{ x_2^k-x_1^k \\ x_3^k-x_2^k \\ \vdots \\ x_n^k -x_{n-1}^k \\ \gamma (y_1^k - L_1x_n^k)  \\ \vdots \\  \gamma (y_m^k - L_m x_n^k)},
\end{equation}
      \hspace*{\algorithmicindent}with $\mathbf{x}^k=(x_1^k,\ldots,x_n^k)\in\Hilbert^n$ and $\mathbf{y}^k=(y_1^k,\ldots,y_m^k)\in\Gilbert_1\times\cdots\times\Gilbert_m$ computed as}
\begin{equation}\label{eq:xkykalg}
 \left\{
 \begin{aligned}
 x_1^k & = J_{A_1}(z_1^k),\\
 x_i^k & = J_{A_i}(z_i^k+x_{i-1}^k-z_{i-1}) \quad \forall i\in{\llbracket 2,n-1\rrbracket}, \\
 x_n^k &= J_{A_n}\left(x_1^k+x_{n-1}^k-z_{n-1}^k -\sum_{j=1}^m L_j^*(\gamma L_j x_1^k - v_j^k)\right), \\
 y_j^k &= J_{B_j/\gamma} \left(L_j(x_1^k+x_n^k) - \frac{v^k_j}{\gamma}\right) \quad \forall j\in{\llbracket 1,m\rrbracket}.
 \end{aligned}
 \right.
\end{equation}
  \EndFor
\end{algorithmic}
\end{algorithm}
\begin{corollary}\label{cor:multiple}
Let $n\geq 2$ and assume that Problem~\ref{problem:PD} has a solution. Let $\lambda\in{]0,1[}$ and $\gamma \in{\left]0,1/ \sum_{j=1}^m \|L_j\|^2 \right]}$. Given some initial points $\mathbf{z}^0= (z_1,\ldots,z_{n-1})\in\Hilbert^{n-1}$ and $\mathbf{v}^0 = (v_1^0,\ldots,v_m^0)\in\Gilbert_1\times\cdots\times\Gilbert_{m}$, consider the sequences $(\mathbf{z}^k,\mathbf{v}^k)_{k\in\mathbb{N}}$ and $(\mathbf{x}^k,\mathbf{y}^k)_{k\in\mathbb{N}}$ generated by Algorithm~\ref{alg:liftedPD}.
 Then, the following assertions hold:
 \begin{enumerate}[(i)]
 \item The sequence $(\mathbf{z}^k,\mathbf{v}^k)_{k\in\mathbb{N}}$ converges weakly to a point $(\barbf{z},\barbf{v})\in\Hilbert^{n-1}\times\Gilbert_1\times\cdots\times\Gilbert_m$.
 \item The sequence $(x_1^k, \ldots, x_n^k,y_1^k,\ldots,y_m^k)_{k\in\mathbb{N}}$ converges weakly to $(\bar{x},\ldots,\bar{x},L_1\bar{x},\ldots,L_m\bar{x})$ with $\bar{x}\in\Hilbert$ solving the primal inclusion~\eqref{eq:primalintroBj}.
 \item For all $i\in{\llbracket 1,n\rrbracket}$, the sequence $(\gamma L_1 x_i^k-v_1^k,\ldots,\gamma L_m x_i^k-v_m^k)_{k\in\mathbb{N}}$  converges weakly to $(\gamma L_1\bar{x}-\bar{v}_1, \ldots, \gamma L_m\bar{x}-\bar{v}_m)$, which solves the dual inclusion~\eqref{eq:dualintroBj}.
 \end{enumerate}
\end{corollary}
\begin{proof}
Just note that Problem~\ref{problem:PD} can be reformulated as an instance of Problem~\eqref{eq:primal}-\eqref{eq:dual} by replacing $B$ by the  operator $\mathbf{B}: \Gilbert_1\times\cdots\times\Gilbert_m\setto\Gilbert_1\times\cdots\times\Gilbert_m$ defined as the cartesian product $\mathbf{B} :=\bigtimes_{j=1}^m B_j$ and $L$ by the linear operator $\mathbf{L}:=\bigtimes_{j=1}^m L_j$. In particular, $\|\mathbf{L}\|^2  =\sum_{j=1}^n\|L_j\|^2$  and its adjoint operator  is $\mathbf{L}^*:\Gilbert_1\times\cdots\times\Gilbert_m\to\Hilbert:(v_1,\ldots,v_m)\to\sum_{j=1}^m L_j^* v_j$. Hence, the result follows by considering the averaged nonexpansive operator $T$ in~\eqref{eq:defT} for this choice of operators and applying Theorem~\ref{t:conv}.
\end{proof}

\subsection{Minimality for primal-dual parametrized resolvent splitting}\label{sect:minimal}

In this section, we adapt the concept of lifted splitting to primal-dual algorithms. First, we extend the definition of fixed point encoding to englobe primal-dual problems. As in Section~\ref{subsect:resolventsplitting}, we denote by $T$ a fixed point operator and by $S$ a solution operator, both parametrized by the maximally monotone operators as well as the linear and adjoint operators appearing in Problem~\ref{problem:PD}.
\begin{definition}[Fixed point encoding]
A pair of operators $(T,S)$ is a \emph{fixed point encoding} for Problem~\ref{problem:PD} if, for all particular instance of the problem,
\begin{equation*}
\Fix{T}\neq \emptyset \Longleftrightarrow \zer{\left(\sum_{i=1}^n A_i+ \sum_{j=1}^m L_j^*B_jL_j\right)}\neq\emptyset \text{ and } \mathbf{z}\in\Fix{T} \Longrightarrow S(\mathbf{z}) \in\mathbf{Z},
\end{equation*}
where we recall that $\mathbf{Z}$ denotes the set of primal-dual solutions of the problem.
\end{definition}
When talking about lifting for primal-dual problems, the need to distinguish between variables in the space of primal solutions and dual solutions arises. This motivates the following definition.
\begin{definition} (Primal-dual lifting)\label{def:pdlifting}
Let $d,f\in\mathbb{N}$. A fixed point encoding $(T,S)$ is a $(d,f)$-\emph{fold lifting} for Problem~\ref{problem:PD} if
\begin{equation*}
T:\Hilbert^{d}\times\Gilbert_{1}^{f_1}\times\cdots\times\Gilbert_m^{f_m} \to\Hilbert^{d}\times\Gilbert_{1}^{f_1}\times\cdots\times\Gilbert_m^{f_m} \end{equation*}
and
\begin{equation*}
S:\Hilbert^{d}\times\Gilbert_{1}^{f_1}\times\cdots\times\Gilbert_m^{f_m}\to\Hilbert\times\Gilbert_1\times\cdots\times\Gilbert_{m},
\end{equation*}
 where $f_j\geq 0 $ for all $i\in{\llbracket 1,m\rrbracket}$ and $f = \sum_{j=1}^m f_j$. We adopt the convention that the space $G_j$ vanishes from the equation when $f_j=0$.
\end{definition}
The need to control the Lipschitz constants of the linear operators requires the introduction of parameters in the resolvents of the maximally monotone operators. This motivates the definition of parametrized resolvent splitting introduced in Section~\ref{subsect:resolventsplitting}  and which we now adapt to primal-dual splitting algorithms.
\begin{definition}[Parametrized primal-dual resolvent splitting]
A  fixed point encoding $(T,S)$ for Problem~\ref{problem:PD} is a \emph{parametrized primal-dual resolvent splitting} if, for all particular instance of the problem, there is a finite procedure that evaluates $T$ and $S$ at a given point which only uses vector addition, scalar multiplication and the parametrized resolvents of $A_1, \ldots A_n$ and $B_1, \ldots, B_m$.
\end{definition}
\begin{definition}[Frugality]
A parametrized primal-dual resolvent splitting $(T,S)$ for Problem~\ref{problem:PD} is \emph{frugal} if, in addition, each of the parametrized resolvents of $A_1,\ldots,A_n$  and $B_1, \ldots, B_m$ is used exactly once.
\end{definition}
\begin{remark}[On the absence of restrictions on the evaluation of the linear operators]
Since in the finite case, a forward evaluation of a linear operator is computationally equivalent to performing vector addition and scalar multiplication, this suggests that for practical applications there is no computational need to control the number of evaluations of the linear operators in the definition of frugality.
\end{remark}
\begin{example}
Let $n\geq2$ and consider Problem~\ref{problem:PD}. Let $T:\Hilbert^{n-1}\times \Gilbert_1\times\cdots\times \Gilbert_m\to\Hilbert^{n-1}\times \Gilbert_1\times\cdots\times \Gilbert_m$ be the operator defined in~\eqref{eq:defT} by setting $B := \bigtimes_{j=1}^m B_j$ and $L:=\bigtimes_{j=1}^m L_j$. Let $S:\Hilbert^{n-1}\times \Gilbert_1\times\cdots\times \Gilbert_m\to\Hilbert\times \Gilbert_1\times\cdots\times \Gilbert_m$ be defined as
\begin{equation*}
S\colvec{\mathbf{z}\\ \mathbf{v}} := \colvec{ J_{A_1}(z_1) \\ \gamma L_1 J_{A_1}(z_1) - v_1 \\\vdots \\ \gamma L_m J_{A_1}(z_1) - v_m }.
\end{equation*}
Then, by Lemma~\ref{lema:ZFix} and Corollary~\ref{cor:multiple}, the pair $(T,S)$ is a frugal parametrized resolvent splitting with $(n-1,m)$-fold lifting.
\end{example}
The following result shows that the lifting of Algorithm~\ref{alg:liftedPD} is minimal among frugal primal-dual parametrized resolvent splitting algorithms with $m$ dual variables.
\begin{theorem}[Minimality theorem for frugal parametrized splitting]
Let $(T,S)$ be a frugal  primal-dual parametrized resolvent splitting for Problem~\ref{problem:PD} with $(d,m)$-fold lifting. Then, if $n\geq 2$, necessarily $d\geq n-1$.
\end{theorem}
\begin{proof}
By way of contradiction, let $(T,S)$ be a frugal parametrized primal-dual resolvent splitting for Problem~\ref{problem:PD} with $(d,m)$ fold lifting and $d < n-1$. Consider the instance of the problem in which $L_j =\Id :\Hilbert \to \Hilbert$ for all $j\in{\llbracket1,m\rrbracket}$. Then, Problem~\ref{problem:PD} becomes the classical monotone inclusion problem with $n+m$ operators and $(T,S)$ is a frugal resolvent splitting with $(d+m)$-fold lifting for such problem with $d+m < n+m-1$, which contradicts Theorem~\ref{th:minimallifting}.
\end{proof}
Finally, we conclude this section by highlighting that Algorithm~\ref{alg:liftedPD} can be applied with $n<2$, by setting $A_i=0$ if required. However, a reduction in the lifting is not obtained in this case.
\begin{remark}[Algorithm~\ref{alg:liftedPD} when $n\leq1$]
Consider Algorithm~\ref{alg:liftedPD}  applied to Problem~\ref{problem:PD} with $n\leq1$. We distinguish the two cases:
\begin{enumerate}[(i)]
\item If $n=1$, then Algorithm~\ref{alg:liftedPD} has $(1,m)$-lifting. Indeed, equations~\eqref{eq:zkvkalg} and~\eqref{eq:xkykalg} become
\begin{equation}
  \colvec{z^{k+1} \\ \mathbf{v}^{k+1}} = \colvec{z^{k} \\ \mathbf{v}^{k}} + \lambda
\colvec{ x^k-z^k \\ \gamma (y_1^k - L_1x^k)  \\ \vdots \\  \gamma (y_m^k - L_m x^k)} \quad \forall k\geq 0,
\end{equation}
 and
\begin{equation}\label{eq:xkykalg2}
 \left\{
 \begin{aligned}
 x^k &= J_{A_1}\left(z^k -\sum_{j=1}^m L_j^*(\gamma L_j z^k - v_j^k)\right), \\
 y_j^k &= J_{B_j/\gamma} \left(L_j(z^k+x^k) - \frac{v^k_j}{\gamma}\right), \quad \forall j\in{\llbracket 1,m\rrbracket},
 \end{aligned}
 \right.
\end{equation}
respectively. This means that, in contrast with what happens when $n\geq2$, there is no reduction in the lifting with respect to the number of operators involved.
\item If $n=0$, the scheme also has $(1,m)$-lifting. In fact, the scheme is the  same as in the previous case but substituting $J_{A_1}$ by $\Id$ in~\eqref{eq:xkykalg2}. Note that this is also the lifting obtained by the already known algorithms in the literature applied to this case.
\end{enumerate}
\end{remark}

\section{Numerical experiments}\label{sect:numerics}

In this section, we test our algorithm for solving an ill-conditioned linear inverse problem which arises in image deblurring and denoising.
Let $b\in\mathbb{R}^n$ be an observed blurred and noisy image of size $M \times N$, with $n=MN$ for grayscale and $n=3MN$ for color images, and denote by $A\in\mathbb{R}^{n\times n}$  the blur operator. The problem can be tackled by means of the regularized convex non-differentiable  problem
\begin{equation}\label{eq:deblur1}
\inf_{s \in \mathbb{R}^n} \left\{ \|A s- b\|_1 + \alpha_1 \|Ws\|_1 + \alpha_2  TV(s) + \delta_{[0,1]^n}(s) \right\},
\end{equation}
where $\alpha_1,\alpha_2 >0$ are regularization parameters, $\delta_{[0,1]^n}$ denotes the indicator function of the set $[0,1]^n$, $TV:\mathbb{R}^n\to\mathbb{R}$ is the discrete isotropic total variation function and  $W$ is the linear  operator given by the normalized \emph{nonstandard Haar transform}~\cite{wavelets}.

Recalling Remark~\ref{remark:gamma}, it is of interest to consider a mechanism which allows tuning the parameter $\gamma$ appearing in the definition of the norm given by the inner product in~\eqref{eq:normgamma} to an appropriate value. To this aim, we perform in~\eqref{eq:deblur1} a change of variable of the form $s=\mu x$, with $\mu >0$, and instead handle the problem
\begin{equation}\label{eq:deblur1.5}
\inf_{x \in \mathbb{R}^n} \left\{ \mu \left\|A x- \frac{b}{\mu}\right\|_1 + \alpha_1\mu \|Wx\|_1 + \alpha_2  TV(\mu x) + \delta_{[0,1/\mu]^n}(x) \right\}.
\end{equation}
Below we will see the way in which the choice of $\mu$ can help setting a suitable parameter~$\gamma$.

The minimization problem in~\eqref{eq:deblur1.5} can be modeled as a composite monotone inclusion problem.
For this, define the operator $L:\mathbb{R}^n\to\mathbb{R}^n\times\mathbb{R}^n : x_{i,j} \to (L_1  x, L_2 x)$ where $L_1$ and $L_2$ are defined component-wise as
\begin{equation}
\begin{aligned}
L_1 x_{i,j} = \left\{
\begin{array}{ll}
\frac{x_{i+1,j} - x_{i,j}}{\mu}, & \text{if } i<M, \\
0,  & \text{otherwise},
\end{array}
\right.
\text{ and }
L_2 x_{i,j} = \left\{
\begin{array}{ll}
\frac{x_{i,j+1} - x_{i,j}}{\mu}, & \text{if } j<N,\\
0, & \text{otherwise}.
\end{array}
\right.
\end{aligned}
\end{equation}
Then the parametrized total variation function can be written as $TV(\mu\;\cdot) = \| L(\cdot)\|_{\times}$, with $\|(p,q)\|_{\times} :=  \sum_{i=1}^m \sum_{j=1}^n \sqrt{p_{i,j}^2 + q_{i,j}^2}$.  Furthermore, an upper bound of the Lipschitz constant of $L$ is given by $\|L\|^2 \leq 8\mu^2 $ (see~\cite{chambolleL} for details).

By~\cite[Proposition~27.5]{bauschke2017}, obtaining a solution to the following problem is equivalent to solving \eqref{eq:deblur1.5}
\begin{equation}\label{eq:deblur2}
\text{find } x\in\mathbb{R}^n \text{ such that } 0 \in \left( N_{[0,1/\mu]^n}+ W^* \circ \partial g_1 \circ W  + A^* \circ \partial g_2 \circ A + L^* \circ \partial g_3 \circ L \right)(x),
\end{equation}
with $g_1: \mathbb{R}^n \to \mathbb{R}$, $g_1(y) = \alpha_1\mu\|y\|_1$, $g_2:\mathbb{R}^n \to \mathbb{R}$, $g_2(y) = \mu\|y - b/\mu\|_1$, $g_3:\mathbb{R}^n \times \mathbb{R}^n \to \mathbb{R}$,
$g_3(p,q) = \alpha_2 \|(p,q)\|_{\times}$, and $N_{[0,1/\mu]^n}$ the normal cone operator to the set ${[0,1/\mu]^n}$.
In order to implement Algorithm~\ref{alg:liftedPD} for solving~\eqref{eq:deblur2}, we need the expression of the following resolvents and proximity operators.
By~\cite[Proposition~23.25~(iii)]{bauschke2017}, the second term in~\eqref{eq:deblur2} is a maximally monotone operator and its resolvent can be expressed as
\begin{equation*}
J_{W^* \circ \partial g_1 \circ W} =  \Id - W^* \circ \left( \Id - \prox_{g_1}\right)\circ W = \Id -  W^*\circ \prox_{g_1^*} \circ W,
\end{equation*}
where $\prox_g=J_{\partial g}$ denotes the proximity operator of a function $g$, and $g_1^*$ is the conjugate function to $g_1$, which is equal to the indicator function $\delta_{[-\alpha_1\mu,\alpha_1\mu]^n}$,    and  thus $\prox_{g_1^*} = P_{[-\alpha_1\mu,\alpha_1\mu]^n}$. Given $\sigma >0$, the proximity operators of $g_2$ and $g_3$ are, respectively,
\begin{equation*}
\prox_{\sigma g_2}(x) = \frac{b}{\mu}+\prox_{\sigma \mu \|\cdot\|_1} \left(x-\frac{b}{\mu}\right) = \frac{b}{\mu} + \sign\left(x-\frac{b}{\mu}\right) \odot \left[\left|x-\frac{b}{\mu}\right| - \sigma \mu \right]_{+},
\end{equation*}
where $\odot$ denotes element-wise product and $[\, \cdot\, ]_+$ and $| \cdot |$ are applied element-wise, and
\begin{equation*}
\begin{aligned}
\prox_{\sigma g_3} &  = \Id - \sigma \prox_{\frac{1}{\sigma} g_3^*} \circ \frac{1}{\sigma} \Id
& = \Id - \sigma P_S \circ \frac{1}{\sigma} \Id,
\end{aligned}
\end{equation*}
since the conjugate function of $g_3$ is $g_3^*:\mathbb{R}^n\times\mathbb{R}^n \to \mathbb{R}^n$, $g_3^* = \delta_S$, with
the set $S$ defined as
\begin{equation*}
S:= \left\{ (p,q)\in\mathbb{R}^n\times\mathbb{R}^n : \max_{1\leq i\leq M, 1\leq j\leq N} \sqrt{p_{i,j}^2 + q_{i,j}^2} \leq \alpha_2 \right\},
\end{equation*}
and the projection operator $P_S:\mathbb{R}^n\times\mathbb{R}^n\to S$ is given component-wise by
\begin{equation*}
(p_{i,j},q_{i,j}) \mapsto \alpha_2 \frac{(p_{i,j},q_{i,j})}{\max{\{\alpha_2, \sqrt{p_{i,j}^2+q_{i,j}^2}\}}}, \quad 1\leq i \leq M, \; 1\leq j\leq N.
\end{equation*}
Hence, when choosing $z^0\in\mathbb{R}^n$, $v_1^0\in\mathbb{R}^n$ and $v_2^0\in\mathbb{R}^n\times\mathbb{R}^n$ as starting values, and letting $\lambda\in{]0,1[}$ and $\gamma\in{\left]0,1/(\|A\|^2+\|L\|^2)\right]}$, the iterative scheme in Algorithm~\ref{alg:liftedPD} becomes
\begin{equation*}
\left\lfloor
\begin{aligned}
x_1^k & = P_{[0,1/\mu]^n} (z^k),  \\
x_2^k & = \left( \Id - W^* \circ P_{[- \alpha_1 \mu,\alpha_1 \mu]^n} \circ W \right) \bigl(2x_1^k-z^k - A^*(\gamma Ax_1^k - v_1^k) - L^*(\gamma Lx_1^k -v_2^k) \bigr), \\
y_1^k & =  \frac{b}{\mu}+\prox_{\frac{\mu}{\gamma} \|\cdot\|_1}\left( A(x_1^k+x_2^k) - \frac{v_1^k}{\gamma} -\frac{b}{\mu}\right), \\
y_2^k & = \left( \Id - \frac{1}{\gamma} P_S\right) \left( \gamma L(x^k_1+x_2^k)- v_2^k\right), \\
z^{k+1} &= z^k + \lambda (x_2^k-x_1^k), \\
v_1^{k+1} &= v_1^k + \lambda \gamma ( y_1^k-A x_2^k), \\
v_2^{k+1} &= v_2^k + \lambda \gamma (y_2^k - L x_2^k).
\end{aligned}
\right.
\end{equation*}
In our experiment, we replicate the problem in~\cite[Section~4.2]{botDR}, where an extensive comparison between different primal-dual algorithms is presented. Since the best performing algorithm is the Douglas--Rachford type primal-dual method in~\cite[Algorithm~3.1]{botDR}, we limit our comparison to this algorithm, whose detailed implementation is given in the cited work. We ran our experiments in Matlab, making use of the inbuilt functions \texttt{fspecial} and \texttt{imfilter} to define an operator $A$ which is a Gaussian blur operator of size $9\times 9$ with standard deviation 4 and reflexive boundary conditions. In particular, $A$ verifies $\|A\|=1$ and $A^*=A$. We employed as observed image $b$ a picture taken at the Sch\"onbrunn Palace Gardens (Vienna) subjected to the already specified blur followed by the addition of a zero-mean Gaussian noise with standard deviation $10^{-3}$ (see Figure~\ref{fig:images}). To test the influence on the performance of the picture size, we resized the original picture to different pixel resolutions (see Table~\ref{tab:comparison}).

When measuring the quality of the restored images, we use the \emph{improvement in signal-to-noise-ratio} (ISNR), which is given by
\begin{equation*}
\mathrm{ISNR}_k = 10\log_{10}\left(\frac{\|x-b\|^2}{\|x-x^k\|^2}\right),
\end{equation*}
where $x$ and $x^k$ are the original and the reconstructed image at iteration $k$, respectively. We tuned the regularization parameters in order to guarantee an adequate ISNR value for the restored images, setting $\alpha_1 := 0.005$ and $\alpha_2:=0.009$.

We recall that the stepsize parameter $\gamma$ of Algorithm~\ref{alg:liftedPD} must be taken in the interval $\gamma \in{\left]0,1/(\|A\|^2+\|L\|^2)\right]}= {\left]0,1/(1+8\mu^2)\right]}$. When $\mu = 1$ (i.e., we solve~\eqref{eq:deblur1}), this interval is $]0,0.111]$. In our numerical experiments we empirically observed that a very small stepsize negatively affects the performance of the algorithm, as mentioned in Remark~\ref{remark:gamma}.  After testing different options, the most convenient one seems to be $\mu = 1/\sqrt{8}$, which implies making the Lipschitz constant of both linear operators in the problem equal to 1.

The initialization of each of the methods was the following:
\begin{itemize}
\item DR1(\cite[Algorithm~3.1]{botDR}): starting points $x_0=b$ and $(v_{1,0},v_{2,0},v_{3,0})=(0,0,0)$, $\sigma_1=1$, $\sigma_2=0.05$, $\sigma_3=0.05$, $\tau = 1(\sigma_1+\sigma_2+8\sigma_3)^{-1}-0.01$, $\lambda_n  =1.5$ for al $n\in\mathbb{N}$.
\item Algorithm~\ref{alg:liftedPD} with $\mu=1$: starting points $z_0 = b$ and $(v_1^0,v_2^0) = (0,0)$, $\lambda = 0.99$ and $\gamma = 1/9$;

\item Algorithm~\ref{alg:liftedPD} with $\mu=1/\sqrt{8}$: starting points $z_0=b/ \mu$ and $(v_1^0,v_2^0) = (0,0)$, $\lambda = 0.99$ and $\gamma = 1/2$.
\end{itemize}

We performed 400 iterations of each of the algorithms and compared the values of the objective function in~\eqref{eq:deblur1.5} and the ISNR with respect to the CPU time, which provides a more realistic comparison than iteration count, since DR1 has a higher computational cost per iteration than Algorithm~\ref{alg:liftedPD}. The tests were run on a desktop of Intel Core i7-4770 CPU 3.40GHz with 32GB RAM, under Windows 10 (64-bit). The algorithms were ran 3 times, once for each of the RGB components of the picture. The evolution in CPU time of adding these 3 values of the objective function and those of the ISNR for the $640\times768$-sized picture are represented in Figure~\ref{fig:plots}, where we observe that Algorithm~\ref{alg:liftedPD} with $\mu=1/\sqrt{8}$  obtains slightly better values than those returned by DR1, but in significantly less time.

\begin{figure}[ht!]\centering
\includegraphics[width=.5\textwidth]{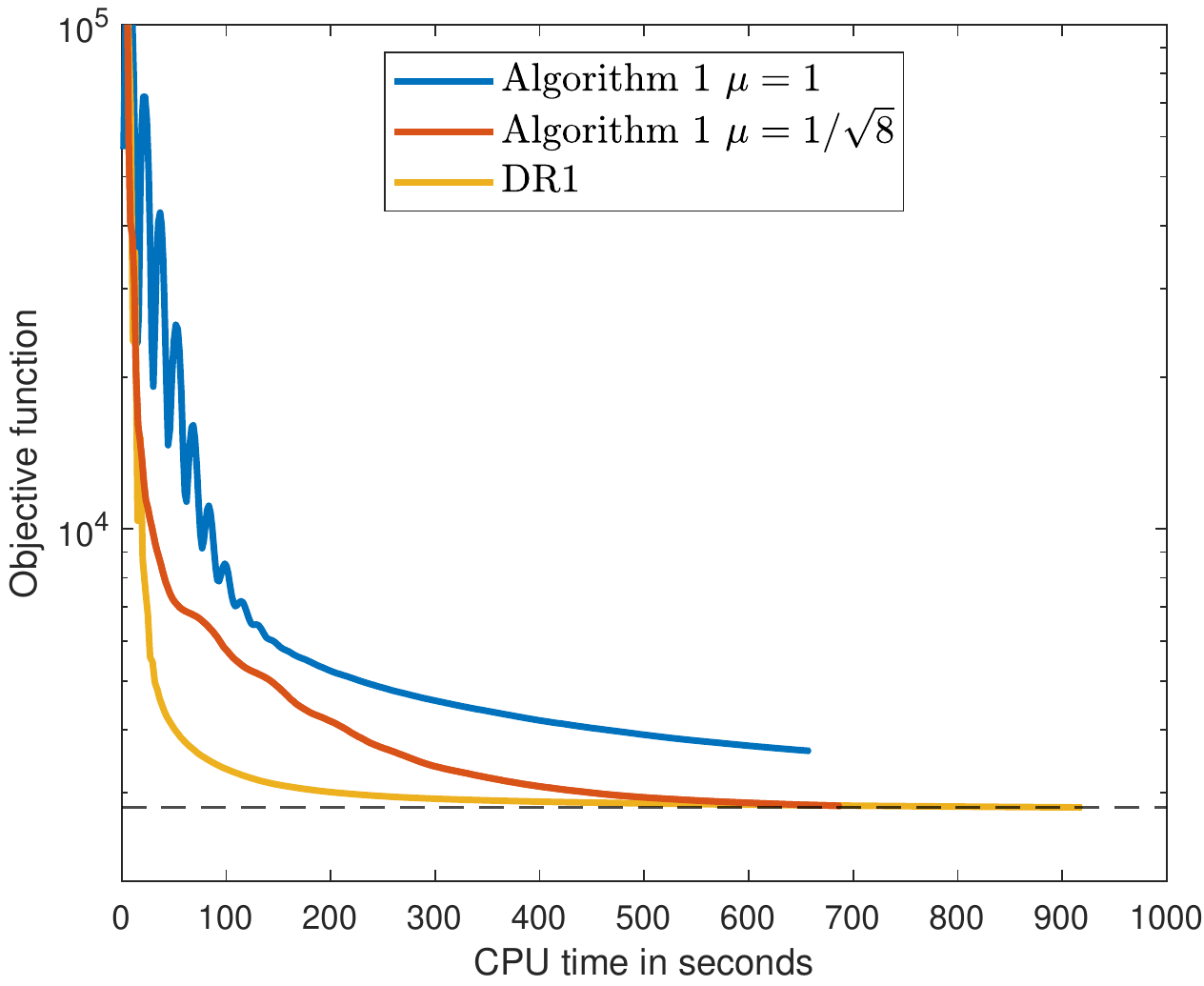}\hfill
\includegraphics[width=.5\textwidth]{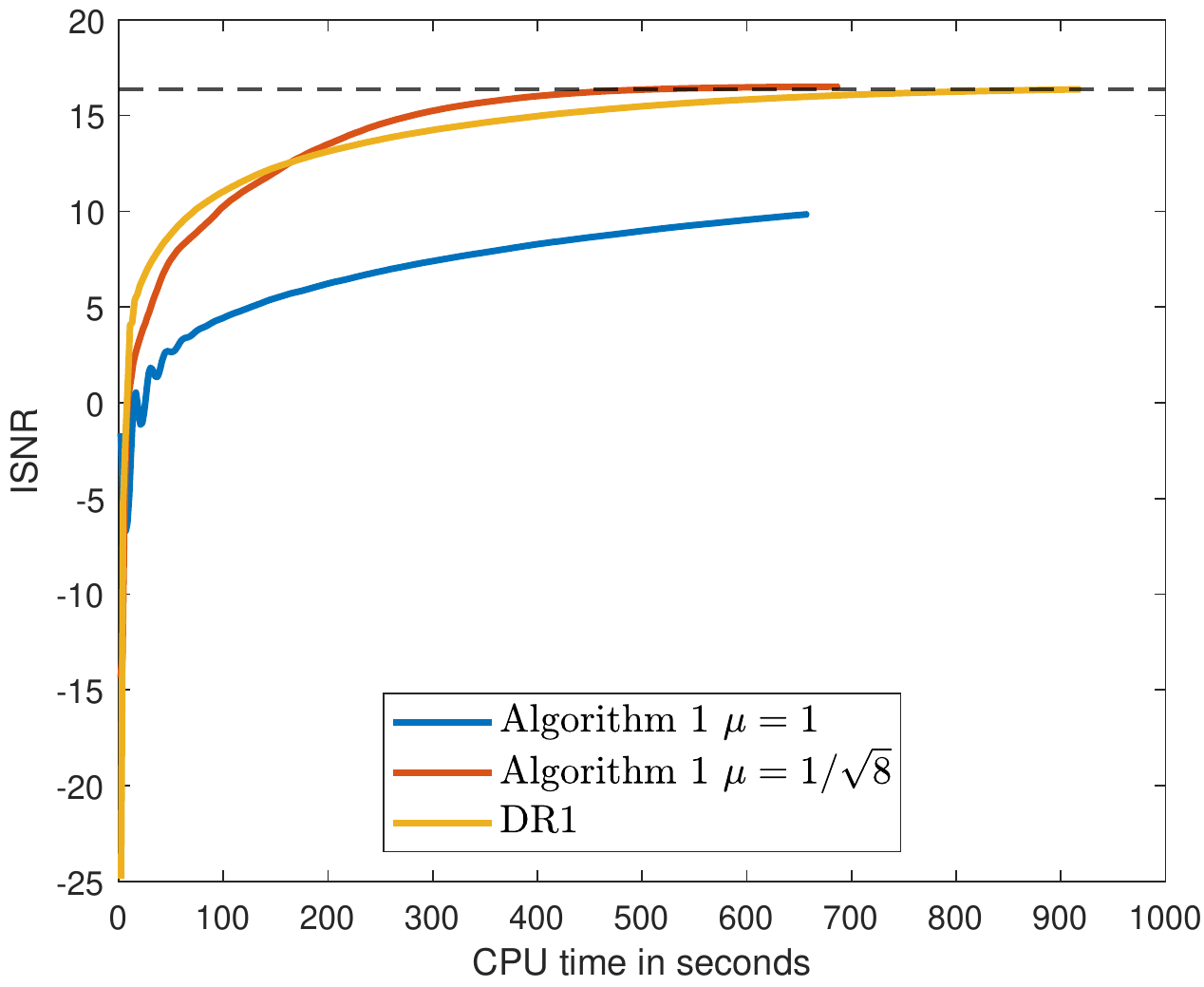}
\caption{The evolution of the values of the objective function and of the ISNR in CPU time for 400 iterations of Algorithm~\ref{alg:liftedPD} with $\mu=1$ and $\mu=1\sqrt{8}$  and DR1, using the $640\times768$ pixels image displayed in Figure~\ref{fig:images}.}\label{fig:plots}
\end{figure}

The restored images are presented in Figure~\ref{fig:images}. There is no much difference between the ones corresponding to Algorithm~\ref{alg:liftedPD} with $\mu=1/\sqrt{8}$ (bottom-middle) and DR1 (bottom-right), but a close look at the image obtained with Algorithm~\ref{alg:liftedPD} with $\mu=1$ permits to observe its worse quality. To show that this trend in the performance of the algorithms is not affected by the image size, we present in Table~\ref{tab:comparison} the results from running the algorithms on the same picture for five different pixel resolutions. Overall, we notice that the CPU time required for computing the 400 iterations is significantly lower for Algorithm~\ref{alg:liftedPD}, as expected. On average, DR1 required 45\% more time than Algorithm~\ref{alg:liftedPD} to compute the 400 iterations, independently of the size of the image. Regarding the parameter $\mu$, Algorithm~\ref{alg:liftedPD} with $\mu=1$ is clearly outperformed by the other two methods, making thus clear the influence that this parameter has on it. The function values obtained were slightly lower for DR1, while the ISNR was slightly lower for Algorithm~\ref{alg:liftedPD} with $\mu=1/\sqrt{8}$, which implies that both algorithms performed similarly with respect to the restored image quality.

\begin{figure}[ht!]\centering
\includegraphics[height=.4\textwidth]{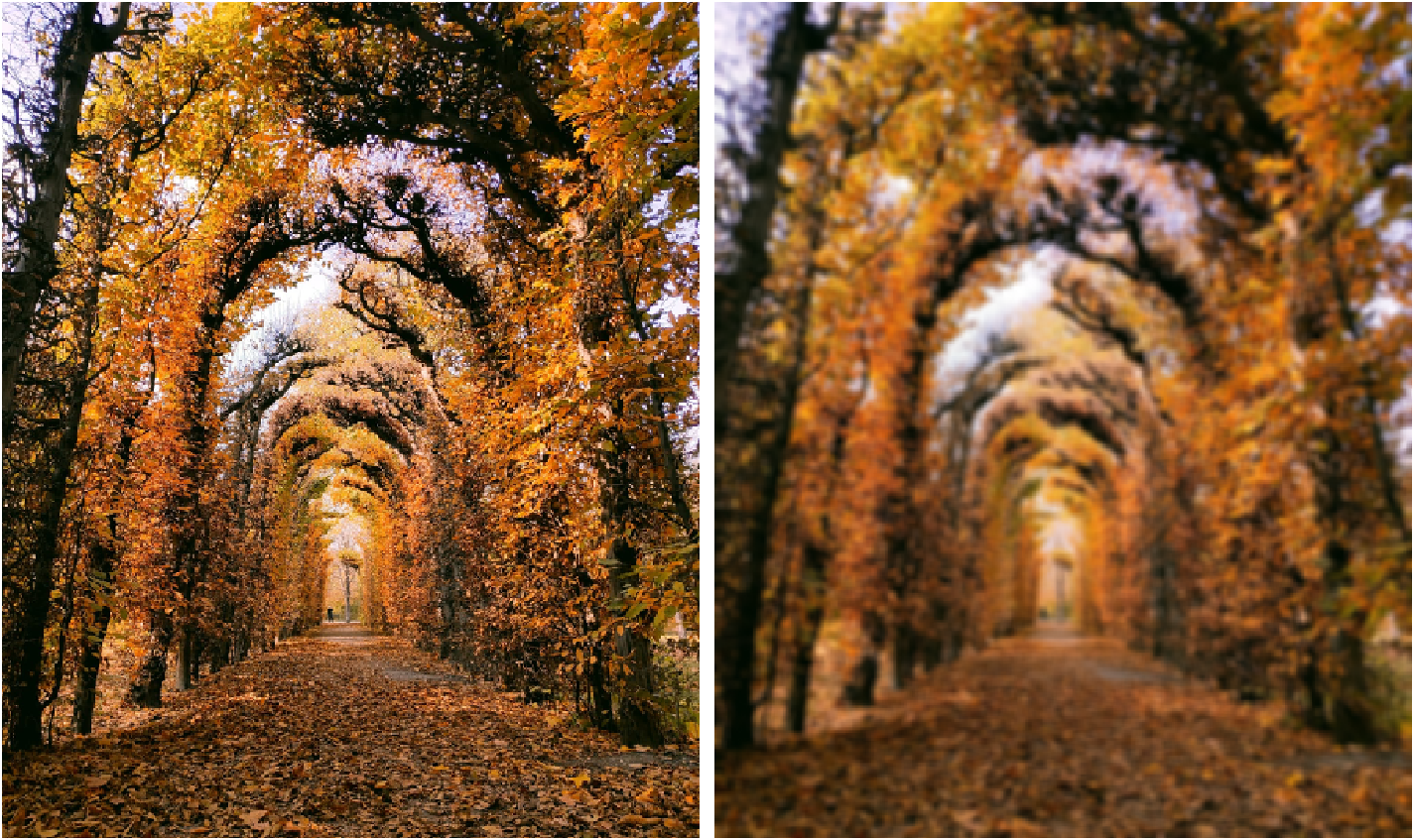}\\
\includegraphics[height=.4\textwidth]{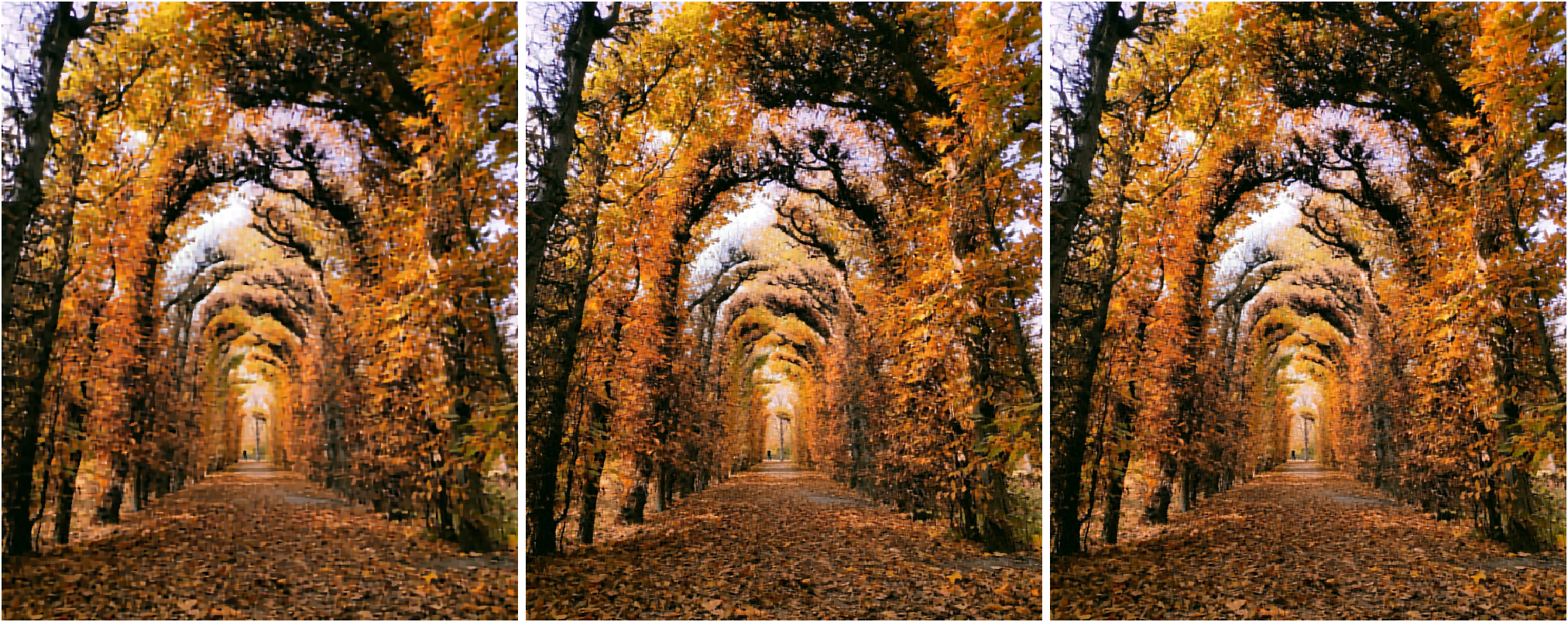}
\caption{On the top, the original $640 \times 768$ pixels image and the blurred and noisy image. On the bottom the images restored after computing 400 iterations of Algorithm~\ref{alg:liftedPD} with $\mu=1$ (left) and $\mu=1/\sqrt{8}$ (middle), and DR1 (right).\label{fig:images}}
\end{figure}

\begin{table}[ht!]\scriptsize\centering
\begin{tabular}{|c|S[table-format=5.1]|S[table-format=5.1]|S[table-format=5.1]||S[table-format=2.1]|S[table-format=2.1]|S[table-format=2.1]||S[table-format=4.1]|S[table-format=4.1]|S[table-format=4.1]|}
\cline{2-10}
\multicolumn{1}{c|}{}&\multicolumn{3}{c||}{Function values} & \multicolumn{3}{c||}{ISNR} & \multicolumn{3}{c|}{CPU time}\\ \cline{1-10}
\multicolumn{1}{|c|}{Resolution}&\multicolumn{1}{c|}{\scriptsize$\mu=1$}& \multicolumn{1}{c|}{\scriptsize$\mu=1/\sqrt{8}$} & \multicolumn{1}{c||}{\scriptsize DR1} & \multicolumn{1}{c|}{\scriptsize$\mu=1$}& \multicolumn{1}{c|}{\scriptsize$\mu=1/\sqrt{8}$} & \multicolumn{1}{c||}{\scriptsize DR1} & \multicolumn{1}{c|}{\scriptsize$\mu=1$}& \multicolumn{1}{c|}{\scriptsize$\mu=1/\sqrt{8}$} & \multicolumn{1}{c|}{\scriptsize DR1}\\ \hline
 $80\times96$  &55.0    &    43.2 &    42.8 & 9.7 & 15.8 & 15.8 &  5.9     & 5.8   & 8.7\\
$160\times 192$ &225.5   &   174.3 &   173.4 & 8.4 & 14.3 & 14.2 & 16.0    &16.2   & 21.1\\
$320\times 384$ &920.3   &   711.2 &   706.0 & 8.7 & 14.9 & 14.8 & 54.7    &51.5   & 74.0\\
$640\times 768$ &3630.3  &  2825.2 &  2804.5 & 9.8 & 16.5 & 16.4 & 294.4   &293.1  & 465.4 \\
$1280\times 1536$ &13084.0 & 10360.0 & 10327.0 &12.8 & 21.0 & 21.0 & 1654.2&1638.5 & 2349.6\\
\hline
\end{tabular}
\caption{Results from running on the picture displayed in Figure~\ref{fig:images} (for various pixel resolutions) 400 iterations of Algorithm~\ref{alg:liftedPD} with $\mu=1$ and $\mu=1/\sqrt{8}$, and DR1.\label{tab:comparison}}
\end{table}

\section{Conclusions and open questions}\label{sect:conclusions}
In this work, we have considered the composite monotone inclusion problem together with its dual counterpart given by Problem~\ref{problem:PD}. We have extended the definition of resolvent splitting given in~\cite{ryu2020} to encompass primal-dual algorithms and the inclusion of parameters in the resolvent and presented a definition of minimal lifting for frugal schemes of this form. We have proposed the first primal-dual algorithm which presents minimal lifting in this sense, and show its good performance with a numerical example.

To conclude, we outline possible directions for further research.

\paragraph{Establishing an optimal criterion for tuning the stepsize $\gamma$:}
We pointed out in Remark~\ref{remark:gamma} the influence that the parameter $\gamma$ can have in the performance of the algorithm. In Section~\ref{sect:numerics} we presented a possibility for controlling this parameter, by making use of a change of variable which modifies the Lipschitz constants of the linear operators, and we empirically showed that it significantly affects the speed of performance of the algorithm. However, there is no guarantee that this strategy is optimal.
It would be interesting to further investigate which is the best way for tuning the value of $\gamma$.

\paragraph{Achieving lifting reduction in the dual variables:} The reduction in the lifting with respect to the number of operators achieved in the algorithm here presented only affects the primal variables. It remains open the question of whether it is possible to reduce the dimension of the underlying space associated to the linearly composed operators. More precisely, if we consider the problem given by
\begin{equation*}
\text{find } x \in \Hilbert \text{ such that  }0 \in\sum_{j=1}^m L^*_j B (L_j x),
\end{equation*}
is it possible to obtain an algorithm for solving this problem with $(0,m-1)$-fold lifting (according to Definition~\ref{def:pdlifting})? Or even with $(1,m-1)$ or $(0,m)$-fold lifting? All these questions remain open.

\paragraph{{\small Acknowledgements}}{\small
\hspace{-2mm}
FJAA and DTB were partially supported by the Ministry of Science, Innovation and Universities of Spain and the European Regional Development Fund (ERDF) of the European Commission, Grant PGC2018-097960-B-C22.
FJAA was partially supported by the Generalitat Valenciana (AICO/2021/165).
RIB was partially supported by FWF (Austrian Science Fund), project P 34922-N.
DTB was supported by MINECO and European Social Fund (PRE2019-090751) under the program ``Ayudas para contratos predoctorales para la formaci\'{o}n de doctores''
2019.

\newpage \appendix

\section{Proof of the minimality theorem for parametrized resolvent splitting}\label{sect:appendix}
Throughout this section, we assume that $n\geq 2$ and we denote by $\mathcal{A}_n$ the set of all $n$-tuples of maximally monotone operators on $\Hilbert$. Hence, an element $A\in\mathcal{A}_n$ is of the form $A = (A_1,\ldots,A_n)$, where $A_i:\Hilbert \setto \Hilbert$ are maximally monotone operators for all $i \in{\llbracket 1,n\rrbracket}$. Every instance of Problem~\ref{eq:monotoneinclusion} is determined by the choice of $A\in\mathcal{A}_n$. In particular, when considering a fixed point encoding for this problem, the fixed point operator and the solution operator are  both parametrized in terms of $A\in\mathcal{A}_n$. To emphasize this idea and to facilitate the exposition, we denote these operators by $T_A$ and $S_A$ in the following.

Let $(T_A,S_A)$ be a $d$-fold lifted frugal parametrized resolvent splitting for Problem~\ref{eq:monotoneinclusion}. By definition, there exists a finite procedure for evaluating $T_A$ and $S_A$ using only vector addition, scalar multiplication and the resolvents $J_{\delta_1 A_1}, \ldots, J_{\delta_n A_n}$ precisely once, where $\delta = (\delta_1,\ldots,\delta_n)^{T}$ is a vector of positive parameters. Following the same reasoning than in~\cite[Section~3]{malitsky2021resolvent}, we can completely describe the evaluation of a point $\mathbf{z} = (z_1,\ldots,z_d)\in\Hilbert^{d}$ by $T_A$ with a series of equations. We directly present them here.
\begin{enumerate}[(i)]
\item There exists $\mathbf{x}= (x_1,\ldots,x_n)\in\Hilbert^{n}$ and $\mathbf{y} = (y_1,\ldots,y_n)\in\Hilbert^{n}$ such that
\begin{equation}\label{eq:Tz1}
 \mathbf{x}= J_{\delta A} (\mathbf{y}) \Longleftrightarrow 0 \in  \mathbf{x} - \mathbf{y} + \delta A(\mathbf{x}),
\end{equation}
where  $\delta A := (\delta_1 A_1, \ldots, \delta_n A_n)\in\mathcal{A}_n$.
\item There exists $Y_z \in \mathbb{R}^{n\times d}$ and a lower-triangular matrix $Y_x \in\mathbb{R}^{n\times n}$ with zeros in the diagonal such that\footnote{Here we make use of an abuse of notation. Indeed~\eqref{eq:Tz2}, should be written as $\mathbf{y} = (Y_z \otimes\Id) \mathbf{z} + (Y_x \otimes \Id) \mathbf{x}$, where $\otimes$ denotes the Kronecker product.}
\begin{equation}\label{eq:Tz2}
\mathbf{y} = Y_z \mathbf{z} + Y_x\mathbf{x}.
\end{equation}
\item By frugality, there exists $T_z\in\mathbb{R}^{d\times d}$ and $T_x\in\mathbb{R}^{d \times n}$ such that
\begin{equation}\label{eq:Tz3}
T_A(\mathbf{z}) = T_z\mathbf{z} + T_x\mathbf{x}.
\end{equation}
\end{enumerate}
Similarly, also by frugality, the evaluation of $\mathbf{z}$ by the solution operator $S$ can be expressed as
\begin{equation}\label{eq:Sz}
S_A(\mathbf{z}) = S_z \mathbf{z} + S_x \mathbf{x},
\end{equation}
where $S_z\in\mathbb{R}^{1\times d}$ and $S_x\in\mathbb{R}^{1\times n}$.

The proof of the next technical lemma can be obtained by following the same steps than in~\cite[Lemma~3.1]{malitsky2021resolvent}, so we do not replicate it here.
\begin{lemma}\label{lemma:M}
Let $(T_A,S_A)$ be a frugal parametrized resolvent splitting for Problem~\ref{eq:monotoneinclusion}. Let $M$ denote the block matrix given by
\begin{equation*}
M :=
\left[
\begin{matrix}
0 &  \Id & -\Id & \delta^{T} \Id \\
Y_z & Y_x & -\Id & 0 \\
T_z-\Id & T_x & 0 & 0
\end{matrix}
\right].
\end{equation*}
If $\mathbf{z}\in\Fix{T_A}$, then there exists $\mathbf{v} = \left[ \mathbf{z}, \mathbf{x}, \mathbf{y}, \mathbf{a} \right]^{T}\in\ker{M}$ with $\mathbf{a}\in A(\mathbf{x})$. Conversely, if $\mathbf{v} = \left[ \mathbf{z}, \mathbf{x}, \mathbf{y}, \mathbf{a} \right]^{T}\in\ker{M}$ and $\mathbf{a}\in A(\mathbf{x})$, then $\mathbf{z}\in\Fix{T_A}$, $\mathbf{x} = J_{\delta A}(\mathbf{y})$ and $S_A(\mathbf{z}) = S_z\mathbf{z} + S_x\mathbf{x}$.
\end{lemma}
\begin{proposition}[Solution operator]\label{prop:solutionmapping}
Let $(T_A,S_A)$ be a frugal parametrized resolvent splitting for Problem~\ref{eq:monotoneinclusion}. Then, for all $\barbf{z}\in\Fix{T_A}$ and $\barbf{x} = J_{\delta A}(\barbf{y})$, we have
\begin{equation}\label{eq:solutionmapping}
S_A(\barbf{z}) = \frac{1}{n} \sum_{i=1}^n (\bar{y}_i -\delta_i \bar{a}_i)= \bar{x}_1 = \cdots = \bar{x}_n,
\end{equation}
where $\barbf{a} =A(\barbf{x})$.
\end{proposition}
\begin{proof}
Consider a particular instance of Problem~\ref{eq:monotoneinclusion} given by some operators $A\in\mathcal{A}_n$. Let $T_A$ and $S_A$ be the fixed point and solution operators of this particular instance, respectively. Let $\barbf{z}\in\Fix{T_A}$ and $x^* = S_A(\barbf{z})$. By Lemma~\ref{lemma:M}, there exists $\mathbf{v} :=  \left[ \mathbf{z}, \mathbf{x}, \mathbf{y}, \mathbf{a} \right]^{T}\in\ker{M}$ with $\barbf{a}\in A(\barbf{x})$ and $x^* = S_A(\barbf{z}) = S_z \barbf{z} + S_z \barbf{x}$.

Consider now the $n+1$ instances of Problem~\ref{eq:monotoneinclusion} given by the $n$-tuples of maximally monotone operators $A^{(0)}, A^{(1)},\ldots, A^{(n)}\in\mathcal{A}_n$ defined as
\begin{equation*}
A^{(0)}(\mathbf{x}) := \barbf{a} \text{ and } A^{(j)}(\mathbf{x}) := \barbf{a} + \left[
\begin{matrix}
0 \\ \vdots \\ x_j - \bar{x}_j \\ \vdots \\ 0
\end{matrix}
\right]
\quad
\forall j \in{\llbracket 1,n\rrbracket}.
\end{equation*}
Since $\mathbf{v}\in\ker M$ and $\barbf{a} = A^{(j)}(\barbf{x})$, for all $j \in{\llbracket 0,n\rrbracket}$, Lemma~\ref{lemma:M} implies that $\barbf{z} \in \Fix{T_{A^{(j)}}}$, $\barbf{x} = J_{\delta A^{(j)}}(\barbf{y})$ and thus, $S_{A^{(j)}}(\barbf{z})= S_z \barbf{z} + S_x\barbf{x}=x^*$ is a solution to every  instance. Therefore, we have $0 = \sum_{i=1}^n A_i^{(0)}(x^*) = \sum_{i=}^n \bar{a}_i$ and hence
\begin{equation*}
0 = \sum_{i=1}^n A^{(j)}_i(x^*) = \sum_{i=1}^n \bar{a}_i + x^* -\bar{x}_j = x^* -\bar{x}_j  \quad \forall j \in{\llbracket 1,n\rrbracket},
\end{equation*}
from where it follows that $x^* = \bar{x}_1 = \cdots = \bar{x}_n$. Finally, since $\barbf{x} = J_{\delta A^{(0)}} (\barbf{y})$, we have that $\barbf{y}-\barbf{x} = \delta A^{(0)}(\barbf{x})=(\delta_1\bar{a}_1,\ldots,\delta_n\bar{a}_n)$. Consequently, $\sum_{i=1}^n \bar{y}_i - n x^* = \sum_{i=1}^n \delta_i \bar{a}_i$, which completes the proof.
\end{proof}
Note that, although the expression for the solution operator given by~\eqref{eq:solutionmapping} differs from the one obtained in~\cite[Proposition~3.2]{malitsky2021resolvent}, it still holds that the vector $\barbf{x}$ belongs to the diagonal subspace of dimension $n$, which we denote by $\Delta_n$. This is what we employ to prove the following theorem.
\begin{theorem}\label{th:minimalliftingappendix}
Let $(T_A,S_A)$ be a frugal parametrized resolvent splitting with $d$-fold lifting for Problem~\ref{eq:monotoneinclusion}. Then $d \geq n-1$.
\end{theorem}
\begin{proof}
Suppose, by contradiction, that $(T_A,S_A)$ is a frugal parametrized resolvent splitting for Problem~\ref{eq:monotoneinclusion} with $d$-fold lifting such that $d \leq n-2$. Consider a particular instance of the problem given by $A\in\mathcal{A}_n$ such that $\zer{\left(\sum_{i=1}^n A_i\right)} \neq \emptyset$ and take $\mathbf{z} \in\Fix T_A$. By Lemma~\ref{lemma:M}, there exists $\mathbf{v}:= \left[ \mathbf{z}, \mathbf{x}, \mathbf{y}, \mathbf{a} \right]^{T}\in \ker M$ with $\mathbf{a} \in A(\mathbf{x})$. The last row of $M$ implies that $0 = (T_z - \Id)\mathbf{z} + T_x\mathbf{x}$. Since $T_x\in\mathbb{R}^{d\times n}$ and $d\leq n-2$, by the rank-nullity theorem, $\dim \ker T_x = n - \dim\rank T_x \geq n - d \geq 2$. Since $\Delta_n $ is a subspace of dimension 1, there exists $\barbf{x}\notin\Delta_n$ such that $T_x \mathbf{x} = T_x\barbf{x}$.

Now, set $\barbf{z} := \mathbf{z}$, $\barbf{y} := Y_z \barbf{z} + Y_x \barbf{x}$ and $\barbf{a} := ((\bar{y}_1-\bar{x}_1)/\delta_1,\ldots,(\bar{y}_n -\bar{x}_n)/\delta_n)$ and consider the instance of the problem given by $\bar{A}\in\mathcal{A}_n$ defined as $\bar{A}(\mathbf{s}) := \barbf{a}$ for all $\mathbf{s}\in\Hilbert^{n}$. Then, $\barbf{v} := \left[ \barbf{z}, \barbf{x}, \barbf{y}, \barbf{a}\right]^{T}\in\ker M$ with $\barbf{a} =\bar{A}(\barbf{x})$. By Lemma~\ref{lemma:M} and Proposition~\ref{prop:solutionmapping}, this implies that $\barbf{x} \in \Delta_{n}$, obtaining thus a contradiction which completes the proof.
\end{proof}
\end{document}